\documentclass[12pt]{article}
\usepackage{fancyhdr}
\usepackage{amssymb}
\usepackage{amsmath}
\usepackage{amsthm}
\usepackage{graphicx}
\usepackage{psfrag}
\usepackage{graphics}
\usepackage{amsthm}
\usepackage{xcolor}
\usepackage{hyperref}

\newcommand{\esp}{\hspace{0.06cm}}
\newcommand{\N}{\mathbb N}

\newcommand{\T}{\mathbb T}

\newcommand{\dist}{V_{\infty}}
\newcommand{\var}{\mathrm{var}}
\newcommand{\R}{\mathbb{R}}
\newcommand{\Z}{\mathbb Z}

\newcommand{\Diff}{\mathrm{Diff}}
\newcommand{\Fix}{\mathrm{Fix}}
\newcommand{\tgamma}{\tilde \gamma}
\newcommand{\eps}{\varepsilon}
\newcommand{\X}{\mathcal{X}}
\newcommand{\Y}{\mathcal{Y}}

\def\res#1{\mathbin{|}{}_{#1}}

\theoremstyle{theorem}

\newtheorem{thm}{Theorem}[section]
\newtheorem{thmintro}{Theorem}

\newtheorem{prop}[thm]{Proposition}
\newtheorem{cor}[thm]{Corollary}
\newtheorem{qs}[thm]{Question}
\newtheorem{lem}[thm]{Lemma}
\theoremstyle{definition}
\newtheorem{rem}[thm]{Remark}

\newtheorem{ex}[thm]{Example}

\begin{document}

\pagestyle{fancy}
\rhead{(Arc-)connectedness for the space of smooth $\Z^d$ actions on 1-manifolds}
\lhead{}

\date{}
\author{H\'el\`ene Eynard-Bontemps \,\, \& \,\, Andr\'es Navas}

\title{(Arc-)connectedness for the space of $\mathbb{Z}^d$ actions\\ 
by $C^2$ diffeomorphisms on 1-manifolds}

\maketitle

\begin{abstract}

We deal with the general problem of connectedness for the space of $\mathbb{Z}^d$ actions by (orientation-preserving) 
diffeomorphisms of a compact 1-manifold. We prove two results. First, the space of $\mathbb{Z}^d$ actions by $C^2$ 
diffeomorphisms of the interval is connected. Second, any two $\mathbb{Z}^d$ actions by $C^2$ diffeomorphisms 
of a compact 1-manifold are connected by a continuous path of $C^{1+\mathrm{ac}}$ actions (where $C^{1+\mathrm{ac}}$ stands for diffeomorphisms 
with absolutely continuous derivative). The latter is the first result of arc-connectedness in regularity larger than $C^1$ in 
this setting. Actually, our proof applies to all $\mathbb{Z}^d$ actions by $C^{1+\mathrm{ac}}$ diffeomorphisms without 
elements with hyperbolic periodic points; the only obstruction to extend it to the general $C^{1+\mathrm{ac}}$ framework comes 
from the failure of the Sternberg-Yoccoz linearization theorem in class $C^{1+\mathrm{ac}}$. 
\end{abstract}

 Centralizers of diffeomorphisms can be viewed as infinitesimal symmetries of a given dynamical system. Starting with the seminal 
 work of Nancy Kopell \cite{kopell}, they have become a central object of study in dynamics. Perhaps the most relevant recent work 
 on this is~\cite{BCW}, which solves a longstanding question of Stephen Smale about centralizers of generic diffeomorphisms and, 
 also, is a wonderful ``window'' to enter into this huge subject. However, despite the effort of many people,  several natural 
problems remain unsolved. Here we deal with a longstanding question raised in the seventies 
by Harold Rosenberg \cite{mio-ICM}, that partly 
inspired the thesis of Jean-Christophe Yoccoz \cite{yoccoz}: {\em Is the space of (orientation-preserving) commuting circle 
diffeomorphisms locally arcwise connected~?} Quite surprisingly, this has revealed as a very difficult question, and only a 
few (and somewhat recent) results in this direction are known. These may be summarized as follows 
(in all that follows, all maps in consideration are assumed to preserve the orientation):

\begin{enumerate}

\item 
The space of $\mathbb{Z}^d$ actions by homeomorphisms of either the interval or the circle 
is arcwise connected (although this seems to be known to the specialists, no written proof 
exists in the literature; see Proposition \ref{connexe-cont} for a short argument);

\item The space of $\mathbb{Z}^d$ actions by $C^{\infty}$ diffeomorphisms of the interval is connected \cite{EB};

\item The space of $\mathbb{Z}^d$ actions by $C^1$ diffeomorphisms of either the circle or the interval is arcwise connected \cite{Na2};

\item Any two $\mathbb{Z}^d$ actions by $C^{1+\mathrm{ac}}$ circle diffeomorphisms may be connected by a path 
of $\mathbb{Z}^d$ actions provided one of the generators acts with an irrational rotation number \cite{Na1}. 
(Here and in all what follows,  $C^{1+\mathrm{ac}}$ stands for $C^1$ maps with absolutely continuous derivative.)

\end{enumerate}

It is worth stressing that results in this domain are very sensitive to different degrees of differentiability. 
The first of the two main results of this paper is a general arc-connectedness result in regularity 
$C^{1+\mathrm{ac}}$, which is the first of this type in regularity higher than $C^1$.

\begin{thmintro}
\label{t:A}
Any two $\mathbb{Z}^d$ actions by $C^2$ diffeomorphisms 
of a compact 1-manifold are connected by a continuous path of $C^{1+\mathrm{ac}}$ actions.
\end{thmintro}

The group $\mathrm{Diff}^{1+\mathrm{ac}}_+(V)$ has a natural topology 
(where $V$ denotes either the circle or the closed interval), namely, the one induced by the metric
$$d (f,g) := \left\| f - g \right\|_{C^1} + \left\| \frac{D^2f}{Df} - \frac{D^2g}{Dg} \right\|_{L^1}.$$
This extends to a metric on the larger group $\mathrm{Diff}^{1+\mathrm{bv}}_+(V)$ of $C^1$ 
diffeomorphisms whose derivative has bounded variation:
$$d(f,g) = \left\| f - g \right\|_{C^1} + \mathrm{var} (\log (Df) - \log (Dg)).$$
We do not deal with continuity issues  in the more general group of piecewise smooths maps, 
since it usually involves problems concerning group topology (see for instance \cite{mann}).

Most of our arguments actually work for $\mathbb{Z}^d$ actions on compact 1-manifolds by $C^{1+\mathrm{ac}}$ diffeomorphisms which 
are not necessarily $C^2$; they only fail when some group element has hyperbolic periodic points, and this is due to the failure of the 
Sternberg-Yoccoz linearization theorem in $C^{1+\mathrm{ac}}$ regularity.
 Our arguments do not give, however, any kind of arcwise 
connectedness in the $C^2$ topology. Actually, this seems to be a very hard problem. Nevertheless, using different techniques, for the case 
of the interval, we prove a general connectedness theorem, which is the analog of item 2. above for actions by $C^2$ diffeomorphisms.

\begin{thmintro}
\label{t:B}
The space of $\mathbb{Z}^d$ actions by $C^2$ diffeomorphisms of the closed interval is connected.
\end{thmintro}

Extending classical work of Joseph Plante 
and William Thurston \cite{PT}, it was proved in \cite{Navas-GAFA,Navas-Mex} that groups of $C^{1+\mathrm{bv}}$ diffeomorphisms 
of 1-manifolds with subexponential growth and, more generally, without free subsemigroups in two generators, are Abelian. Therefore, 
our Theorems \ref{t:A} and \ref{t:B} apply to them.

\vspace{0.1cm}

If $V$ denotes either the circle or the closed interval, actions as those concerned 
by Theorems \ref{t:A} and \ref{t:B} arise as holonomy representations of codimension-1 
foliations of $\T^d\times V$ transverse to the second factor (and tangent to the boundary if nonempty), where $\T^d$ 
denotes the $d$-dimensional torus. The previous statements can thus be translated in terms of foliations as follows:

\vspace{0.4cm}

\noindent{\bf Corollary.} {\em Any two codimension-1 foliations of class $C^2$ of $\T^d\times V$ transverse to the second 
factor (and tangent to the boundary if nonempty) can be connected by a path of $C^{1+\mathrm{ac}}$ foliations. Moreover, 
in the case of  $\T^d\times [0,1]$, they cannot be separated by disjoint open sets of such foliations.}


\section{Discussion and plan of the proofs}

\subsection{On the technique of proof of Theorem \ref{t:A}}
\label{continuous-case} 
Theorem \ref{t:A} concerns both the circle and the interval. However, part of the circle case has been already settled. Indeed, path 
connexion between any action for which a generator has irrational rotation number and the corresponding action by rotations  
follows from item 4. above, which essentially corresponds to \cite[Theorem B]{Na1}  (though some slight modifications in the proof are 
necessary). Moreover, the case where all generators  have a rational rotation number can be reduced to that of the interval by passing 
to a finite index subgroup (yet this reduction requires several extra arguments). Details will be provided later on.

Thus, our main contribution concerns the difficult case of the interval. In this situation, many fixed points in the interior may arise for the action. 
However, the restriction to each connected component of the complement of the set of fixed points is an action with no global fixed 
point. It is hence natural to first deal with (non necessarily faithful) actions of this type, and later check that certain paths of 
deformations supported on disjoint intervals fit nicely provided a good control for them can be ensured.

In case of absence of global fixed points in the interior, in virtue of the famous Kopell's Lemma \cite{kopell}, 
nontrivial elements actually have no fixed points in the interior, and a key role is played by the Mather invariant of these elements. Recall that this captures the 
failure of a diffeomorphism to arise from a $C^1$ vector field of the closed interval. (A review of this appears in \S \ref{section-review}, 
with proofs of new results in the Appendix I.) The discussion then splits into two different subcases. 

\vspace{0.3cm}

\noindent $\bullet$ {\bf The Mather invariants of group elements are trivial.} 

\vspace{0.2cm}

If the endpoints are parabolic fixed points, then we can deform the given action into the trivial one through conjugates. 
This uses the ideas developed in \cite{EN} based on the notion of asymptotic variation introduced in \cite{Na1}\footnote{This was called the {\em asymptotic distortion} in \cite{EN} and \cite{Na1}, 
but we think that asymptotic variation is a better terminology for the purposes of this work.}
. Indeed, 
triviality of the Mather invariant is equivalent to the vanishing of the asymptotic variation, which is the key (and 
necessary) ingredient for implementing the conjugacy argument inside the group of $C^{1+\mathrm{ac}}$ diffeomorphisms.

If one or both of the endpoints is hyperbolic, then it is natural to use the classical Sternberg linearization theorem (in Yoccoz' improved 
version) to transform the original action to one arising from a vector field for which the flow consists of diffeomorphisms that are affine 
close to these endpoints. The idea is then to locally deform the vector field to one which yields parabolic diffeomorphisms, so that 
we can apply the previous argument. This actually works, but requires a more subtle strategy because of the lack 
of a good control on the conjugators whenever the maps become less and less hyperbolic. More precisely, we use 
conjugacies from outside of the group of $C^1$ diffeomorphisms which nevertheless preserve the smooth structure of 
the original maps (this idea comes from \cite{Na2}). 

It is worth stressing that this use of Sternberg's theorem is the only issue where we need to assume that the original action 
is by $C^2$ diffeomorphisms (Sternberg's result is no longer valid for $C^{1+\mathrm{ac}}$ diffeomorphisms; see the Appendix II). All the other 
arguments are valid in $C^{1+\mathrm{ac}}$ regularity. In any case, the proof requires extending to $C^{1+\mathrm{bv}}$ diffeomorphisms the 
classical Szekeres construction of generating vector fields (that is, vector fields whose time-1 map is the underlying diffeomorphism),
as well as some results from \cite{EN} to this context. This is carried out in the Appendix I.

\vspace{0.3cm}

\noindent $\bullet$ {\bf The Mather invariant of a group element is nontrivial.} 

\vspace{0.2cm}

This assumption necessarily implies that the image group is isomorphic to $\mathbb{Z}$; see \S \ref{section-mather}. Denoting by $f$ its generator, 
one is tempted to just deform $f$ by a simple linear homotopy (of its graph), and simultaneously deform the whole action in a coherent way 
(given that every group element is nothing but a power of $f$). However, the major difficulty comes from that $f$ may have a very 
large $C^{1+\mathrm{bv}}$-norm even in the case where the norms of the generators are small. 
(Here and in what follows, by {\em $C^{1+\mathrm{bv}}$-norm} we mean the total variation of the logarithm of the derivative, 
which for a $C^{1+\mathrm{ac}}$ diffeomorphism $f$ will be referred to as the $C^{1+\mathrm{ac}}$-norm and corresponds to the $L^1$ 
norm of its affine derivative $D^2 f / D f$.)
This phenomenon is at the core of the classical  examples of Sergeraert \cite{sergeraert} 
(see \cite{eynard1,eynard2} for recent developments on this), and represents a major obstacle to deform a given action in a controlled way. 
To overcome this difficulty, the new key idea consists in using the equivariance properties of the asymptotic variation and conjugacies in 
order to first conjugate the original action into another one for which we can ensure that the norm of the corresponding (conjugate) 
diffeomorphism $f$ is small, and later proceed to the deformation by homotopy. Informally speaking, we first need to put the 
action in ``good coordinates'' so that the homotopy deformation behaves tamely.

\vspace{0.3cm}

We develop the arguments for each subcase above in the separate sections \S  \ref{section-trivial} and \S \ref{section-no-trivial}. 
The proof of Theorem \ref{t:A} is then concluded in \S \ref{section-todo}, where we carefully put all the pieces together. It is worth 
mentioning that this is not at all straightforward; in particular, several of the aforementioned estimates (as those arising 
in the case of a nontrivial Mather invariant) will be crucial at this step. 

\vspace{0.4cm}


\noindent{\bf The main idea: averaging actions.} 
The strategy of proof above may seem somewhat cryptic since it is described in technical terms. Nevertheless, we would like to stress the main idea, 
which consists (whenever possible) in conjugating the original action by a classical averaging procedure so that it becomes closer and closer 
to an action by isometries. In the present $C^{1+\mathrm{ac}}$ setting, this is achieved by using the affine derivative. For actions by $C^1$ 
diffeomorphisms, the same idea was implemented in \cite{Na2} via the logarithmic derivative $\log D (\cdot)$. For completeness of this work, 
below we give an elementary result in the continuous framework for which the proof uses the same strategy. (Compare 
\cite[Proposition (2.2), Chapitre VII]{herman}.)

\vspace{0.2cm}

\begin{prop} \label{connexe-cont}
The space of $\mathbb{Z}^d$ actions by homeomorphisms of either the interval or the circle is arcwise connected.
\end{prop}

\begin{proof} 
By identifying the endpoints, the case of the interval can be deduced from that of the circle, so let us only consider 
this one.\footnote{The case of the interval can be also ruled out using the classical Alexander trick. Notice that this 
work for any group action by homeomorphisms of the interval, but it doesn't work for actions on the circle, even in 
the Abelian case. Also notice that this argument cannot be applied in higher regularity.} 
Let $f_1, \ldots, f_d$ be the images of the canonical generators of $\Z^d$, 
and let $F_i$ be a lift of $f_i$ to the real line. Denote 
$$B (n) := \{ F_{1}^{n_1} F_2^{n_2} \cdots F_d^{n_{d}}: \,\, 0 \leq n_i < n \},$$  
and consider the map $\varphi_n$ defined as 
$$\varphi_n (x) := \frac{1}{n^d} \sum_{F \in B(n)} F (x).$$
Notice that $\varphi_n$ is a homeomorphism, since it is continuous and strictly increasing. 
Since the maps $F_j$ commute, for each $F_i$ we have 
$$
\varphi_n (F_i (x)) 
= \frac{1}{n^d} \sum_{0 \leq n_j < n} F_1^{n_1} \cdots F_{i-1}^{n_{i-1}}  F_i^{1+n_i}  F_{i+1}^{n_{i+1}} \cdots F_d^{n_d} (x),$$
and, again by commutativity, this equals 
$$\varphi_n (x) + \frac{1}{n^d} \Big[ \sum_{\substack{0 \leq n_j < n \\ j \neq i}} F_i^n (F_1^{n_1} \cdots F_{i-1}^{n_{i-1}}  F_{i+1}^{n_{i+1}} \cdots F_d^{n_d} (x))  
    - F_1^{n_1} \cdots F_{i-1}^{n_{i-1}}  F_{i+1}^{n_{i+1}} \cdots F_d^{n_d} (x) \Big].$$   
Recall that\, $ ( F^n_i (y) - y ) / n$ \, uniformly converges to the translation number $\rho (F_i)$. Since there are $n^{d-1}$ terms of type 
\, $F_i^n (y) - y$ \, in the right-side expression above, we deduce the (uniform) convergence 
$$\varphi_n (F_i (x)) - \varphi_n (x) \longrightarrow \rho (F_i) \qquad \mbox{ as } \quad n \to \infty.$$
Changing $x$ by $\varphi_n^{-1} (x)$, this yields 
$$\varphi_n (F_i (\varphi_n^{-1}(x))) \longrightarrow x +  \rho (F_i) \qquad \mbox{ as } \quad n \to \infty.$$
One readily checks that $\varphi_n$ commutes with the integer translations, hence induces a circle homeomorphism, that we still denote 
by $\varphi_n$. The convergence above translates into that $\varphi_n f_i \varphi_n^{-1}$ uniformly converges to the rotation by $\rho (F_i)$ 
mod. $\! \mathbb{Z}$, which is nothing but the rotation number of $f_i$. We have thus produced a sequence of conjugate actions that 
uniformly converges to an action by rotations. One can then construct a continuous path of such conjugates just by linear interpolation. 
More precisely, one considers circle homeomorphisms of the form $(1-s) \varphi_n + s \, \varphi_{n+1}$, with $s \in [0,1]$. 
Finally, having produced continuous paths of  conjugate actions ending at actions by rotations, one can connect any two 
$\mathbb{Z}^d$ actions just by moving the angles of these rotation actions.
\end{proof}

\begin{rem} The proof above actually shows more: the natural inclusion of $\mathrm{SO}(2,\mathbb{R})^d$ in the space of $\mathbb{Z}^d$ actions 
on the circle (endowed with the compact-open topology) is a homotopy equivalence. (Compare \cite[Proposition 4.2]{ghys}.) We do not know whether 
this result extends to higher regularity.
\end{rem}

\begin{rem} One can produce a different proof of Proposition \ref{connexe-cont} by conjugating as in \cite{Na2} via quasi-invariant probability measures 
that evolve towards the Lebesgue measure. Both arguments apply more generally to nilpotent group actions. The structural results from \cite{parkhe} 
can also be adapted to produce still another proof.
\end{rem}


\subsection{On the technique of proof of Theorem \ref{t:B}}
\label{th-B}

The proof of Theorem \ref{t:B} is identical for $d=2$ or $d > 2$. For readability reasons, we will thus restrict to the case $d=2$. 
The space under scrutiny is hence identified to the subspace of $(\mathrm{Diff}^2_+[0,1])^2$ made of pairs of commuting 
diffeomorphisms, endowed with the $C^2$-topology. 

Roughly, we show that, given a pair of commuting $C^2$ diffeomorphisms $(f_1,f_2)$ and 
$\varepsilon>0$, the interval $[0,1]$ can be subdivided into finitely many intervals $I$ so that:

\begin{itemize}

\item $f_1$ and $f_2$ are $C^2$-tangent to the identity 
at every interior endpoint of each $I$;

\item for each such $I$, the pair of restrictions $(f_1 |_I , f_2 |_I )$ is either $\varepsilon$-close (in $C^2$ topology) 
to $(\mathrm{id}_{I},\mathrm{id}_{I})$ or in the connected component of $(\mathrm{id}_{I},\mathrm{id}_{I})$ among pairs 
of commuting diffeomorphisms that are simultaneously $C^2$ tangent to the identity at each point of $\partial I\cap (0,1)$. 

\end{itemize}

\noindent This shows that, for every $\varepsilon > 0$, the pair $(f_1,f_2)$ is at distance less than $\eps$ from the connected 
component of $(\mathrm{id},\mathrm{id})$ in the space of $C^2$ commuting diffeomorphisms. Since this component is closed, 
$(f_1,f_2)$ belongs to it, thus yielding the announced connectedness. 

Intervals $I$ of the second type are in fact themselves subdivided into two types of intervals depending on whether the relative translation 
number between $f_1$ and $f_2$ is rational or not. In the former case, we show that $(f_1 |_I , f_2 |_I )$ belongs to the \emph{path}-connected 
component of $(\mathrm{id}_{I},\mathrm{id}_{I})$, while in the latter case we show that the closure of the conjugacy class of $(f_1 |_I , f_2 |_I )$ 
(and thus of its path-connected component) intersects the path-connected component of $(\mathrm{id}_{I},\mathrm{id}_{I})$. 
Details are provided in \S \ref{s:B}.

\vspace{0.1cm}

Notice that despite the similarities of the sketch of proof above with that of the $C^{\infty}$ case given in \cite{EB}, there is a huge difference. 
Namely, therein it is proved that the path-connected component of $(\mathrm{id},\mathrm{id})$ is dense in the space of all $\mathbb{Z}^2$ 
actions, so its connected component is the whole space. In the present $C^2$ setting we are unable to decide whether this is true or not. 
The problem is
 that, in the last of the cases above, we do not know whether the approximation of the path-connected 
component of $(\mathrm{id}_{I},\mathrm{id}_{I})$ by conjugates occurs along {\em a path} of conjugates.

In view of this, the next two questions become natural and worth for future research.

\begin{qs} Is the path-connected component of the trivial action dense in the space of all Abelian group actions by $C^2$ diffeomorphisms of the interval?
\end{qs}

\begin{qs} In the same setting, does the path-connected component of any action contain the trivial one in its closure? Is it dense in the space of all actions?
\end{qs}


\section{On the asymptotic variation and Mather invariant}
\label{section-review}

\subsection{Asymptotic variation and conjugacies}
\label{section-distortion}

Given a diffeomorphism $f$ of a compact (connected) 1-manifold $V$ 
({\em i.e.} the circle or the closed interval), we let $\var (\log Df)$ be the total variation 
of the logarithm of its derivative. (We use the notations $C^{1+\mathrm{bv}}$ and $\mathrm{Diff}^{1+\mathrm{bv}}$ 
to refer to maps for which this is a finite number.) The 
{\em asymptotic variation} of $f$ is defined as
$$V_{\infty} (f) := \lim_{n \to \infty} \frac{\var (\log Df^n)}{n}.$$
Notice that this limit exists because of the subadditivity of \esp $\var (\log D(\cdot))$. \esp 
Moreover, for each integer $n \geq 1$, one has 
\, $\var (\log Df^n) \leq n \, \var (\log Df),$ \, 
and therefore 
\begin{equation}\label{eq-control}
\dist (f) \leq \var (\log Df).
\end{equation}
Also notice that \esp $\var \log D (\cdot)$ \esp is invariant under passing to the inverse; as a consequence, 
\begin{equation}\label{eq-inv}
\dist (f) = \dist (f^{-1}).
\end{equation}
Moreover, it is homogeneous: for $n \in \mathbb{Z}$,
\begin{equation}\label{eq-hom}
\dist (f^n) = |n| \, \dist (f).
\end{equation}

Unlike $\var (\log D (\cdot))$, the quantity $\dist (\cdot)$ is invariant under conjugacy. 
Actually, it arises as an infimum along conjugates: for every $C^{1+\mathrm{bv}}$ diffeomorphism, one has 
\begin{equation}\label{eq-conj}
\dist (f) = \inf_{h \in \Diff^{1+\mathrm{bv}}_+(V)} \var (\log D (hfh^{-1})).
\end{equation}
This appears as Proposition 1.2 in \cite{EN} for the case of the interval, yet the very same proof applies to the case of the circle.

Because of the equality above, asymptotic variation is crucial in regard to the problem of approximating either the identity 
(in the case of the interval) or a rotation (in the case of the circle) by conjugates in the $C^{1+\mathrm{bv}}$ topology. This is 
reflected by the next result, which appears as Theorem \ref{t:B} in \cite{Na1} for the case of the circle, but whose proof works verbatim 
for the case of the interval (the relevant hypothesis is the vanishing of the asymptotic variation of maps); 
see Proposition \ref{prop-two} below for more details.

\medskip

\begin{prop} \label{prop-one} 
Let $f_1, \ldots, f_d$ be commuting $C^{1+\mathrm{bv}}$ diffeomorphisms of a compact 1-mani\-fold. 
If the asymptotic variation of each $f_i$ vanishes, then there exists a continuous path (for the $C^{1+\mathrm{bv}}$ topology)
of simultaneous conjugates $h_t f_i h_t^{-1}$ that starts at the given $f_i$ and finishes at isometries. Moreover,
along this path, each function $t \mapsto \var (\log D (h_t f_it_t^{-1}))$ is bounded from above by $\var  (\log Df_i)$.  
Finally, if $f_1,\ldots,f_d$ are of class $C^{1+\mathrm{ac}}$, then the path is continuous for the $C^{1+\mathrm{ac}}$-topology.
\end{prop}

\medskip

The hypothesis of vanishing asymptotic 
variation is satisfied in two relevant cases. On the one hand, it is shown in \cite{Na1} that it holds for every $C^{1+\mathrm{ac}}$ 
circle diffeomorphism of irrational rotation number. (Warning: this result is false for $C^{1+\mathrm{bv}}$ diffeomorphisms; 
see Proposition~2.2 therein.) On the other hand, it also holds if $f$ is a $C^{1+\mathrm{bv}}$ diffeomorphism of the interval 
with no interior fixed point that has a $C^1$ centralizer non-isomorphic to $\mathbb{Z}$ and for which the 
endpoints are parabolic. This follows from the relation between $\dist$ and the Mather invariant, as explained later on.

We next give a more general version of Proposition \ref{prop-one} whose proof follows the very same lines but still applies 
in  case of nonvanishing asymptotic variation.

\medskip

\begin{prop} \label{prop-two} 
Given any family of commuting $C^{1+\mathrm{bv}}$ (resp. $C^{1+\mathrm{ac}}$) 
diffeomorphisms $f_1, \ldots, f_d$  of a compact 1-manifold and $\varepsilon > 0$, 
there exists a $C^{1+\mathrm{bv}}$-continuous (resp. $C^{1+\mathrm{ac}}$-continuous) 
path of simultaneous conjugates $h_t f_i h_t^{-1}$ that starts at the given $f_i$ and finishes at 
(commuting) diffeomorphisms $\bar{f}_i$ such that \, $\var (\log D \bar{f}_i) \leq \dist (f_i) + \varepsilon$. 
Moreover,
 along 
this path, each function $t \mapsto \var (\log D (h_t f_it_t^{-1}))$ is bounded from above by $\var  (\log Df_i)$.
\end{prop}

\begin{proof} For each $n \geq 1$, let $g_n$ be defined by letting $g_n (0) = 0$ and 
\begin{equation}\label{eq:def-g_n}
Dg_n (x) := 
\frac{\left[ \prod_{0 \leq n_j < n} D(f_1^{n_1} \cdots f_d^{n_d}) (x) \right]^{\frac{1}{n^d}}}
{\int_0^1  \left[ \prod_{0 \leq n_j < n} D(f_1^{n_1} \cdots f_d^{n_d}) (y) \right]^{\frac{1}{n^d}} \, dy}.
\end{equation}
This defines a diffeomorphism, since the prescribed value for $Dg_n$ is everywhere positive and 
the total integral of this function equals $1$. Using commutativity and the chain rule 
$$D (g_n \circ f_i \circ g_n^{-1}) (g_n (x)) 
= \frac{Dg_n (f_i (x))}{Dg_n (x)} \cdot Df_i (x) ,$$
we compute:
\begin{eqnarray*}
D (g_n \circ f_i \circ g_n^{-1}) (g_n (x)) 
&=& \frac{\left[ \prod_{0 \leq n_j < n} D(f_1^{n_1} \cdots f_d^{n_d}) (f_i(x)) \right]^{\frac{1}{n^d}}}
{\left[ \prod_{0 \leq n_j < n} D(f_1^{n_1} \cdots f_d^{n_d}) (x) \right]^{\frac{1}{n^d}}} \cdot Df_i (x) \\
&=& \frac{\left[ \prod_{0 \leq n_j < n} D(f_1^{n_1} \cdots f_d^{n_d}) (f_i(x)) \cdot Df_i (x)\right]^{\frac{1}{n^d}}}
{\left[ \prod_{0 \leq n_j < n} D(f_1^{n_1} \cdots f_d^{n_d}) (x) \right]^{\frac{1}{n^d}}} \\
&=& \frac{\left[ \prod_{0 \leq n_j < n} D(f_1^{n_1} \cdots f_i^{1+n_i} \cdots f_d^{n_d}) (x) \right]^{\frac{1}{n^d}}}
{\left[ \prod_{0 \leq n_j < n} D(f_1^{n_1} \cdots f_i^{n_i} \cdots f_d^{n_d}) (x) \right]^{\frac{1}{n^d}}} \\
&=& \frac{\left[ \prod_{0 \leq n_j < n; \, j \neq i} D(f_i^{n} f_1^{n_1} \cdots f_{i-1}^{n_{i-1}} f_{i+1}^{n_{i+1}} \cdots f_d^{n_d}) (x) \right]^{\frac{1}{n^d}}}
{\left[ \prod_{0 \leq n_j < n; \, j \neq i} D(f_1^{n_1} \cdots f_{i-1}^{n_{i-1}} f_{i+1}^{n_{i+1}} \cdots f_d^{n_d}) (x) \right]^{\frac{1}{n^d}}} \\
\end{eqnarray*}
Thus,
$$D (g_n \circ f_i \circ g_n^{-1}) (g_n (x)) 
=\left[ \prod_{0 \leq n_j < n; j \neq i} D(f_i^{n}) (f_1^{n_1} \cdots f_{i-1}^{n_{i-1}} f_{i+1}^{n_{i+1}} \cdots f_d^{n_d} (x) ) \right]^{\frac{1}{n^d}}
,$$
and therefore
$$\log \big( D (g_n \circ f_i \circ g_n^{-1}) (g_n (x)) \big) = 
\frac{1}{n^d} \sum_{\substack{0 \leq n_j < n; \\ j \neq i}} \log (Df_i^n) (f_1^{n_1} \cdots f_{i-1}^{n_{i-1}} f_{i+1}^{n_{i+1}} \cdots f_d^{n_d} (x) ).$$
Since $\var (\log D (\cdot))$ is invariant under change of coordinates, a triangle inequality yields 
$$\var \big( \log D (g_n \circ f_i \circ g_n^{-1})  \big) 
\leq \frac{1}{n^d} \sum_{\substack{0 \leq n_j < n \\ j \neq i}} \var (\log Df_i^n) .$$
Finally, by an elementary counting argument, 
$$\var \big( \log D (g_n \circ f_i \circ g_n^{-1})  \big) 
\leq  \frac{n^{d-1} }{n^d}\, \var (\log Df_i^n) = \frac{\var (\log Df_i^n)}{n}.$$

Now, by definition, the right-side expression above converges to $\dist (f_i)$. Therefore, we may 
fix an integer $N_i$ so that it becomes smaller than or equal to \, $ \dist (f_i) + \varepsilon$. \, 
Letting $N := \max_i N_i$, we obtain a sequence of conjugate 
actions with the desired properties ending at the conjugate action by $g_N$. To obtain a continuous path, it suffices to 
interpolate between (the derivatives) of $g_n$ and $g_{n+1}$: for $s \in [0,1]$, define $g_s$ by letting $g_s(0)=0$ and
$$Dg_s (x) = C_s \, Dg_n (x)^{1-s} \, Dg_{n+1} (x)^s$$
for a well-chosen constant $C_s$ so that  
\, $\int_0^1 Dg_s (x) \, dx = 1.$ \, 
Indeed, this does not increase the $C^{1+\mathrm{ac}}$-norm beyond those of $g_n  f_i  g_n^{-1}$ and $g_{n+1} f_i g_{n+1}^{-1}$. 
The details are left to the reader.
\end{proof}


\subsection{A detour on drift of cocycles in Banach spaces}

Most of the analysis done in \cite{Na1} leading to Proposition \ref{prop-one} works for cocycles with respect to 
isometric actions on Banach spaces (see Lemma 2.1 therein). In the same way as Proposition~\ref{prop-two} extends Proposition 
\ref{prop-one} to the case of nonvanishing asymptotic variation, Lemma~2.1 from \cite{Na1} can be extended to cocycles with
 nonzero drift, as we explain below.
 
Let $U$ be a linear isometric action of a group $\Gamma$ on a Banach space $\mathbb{B}$. A {\em cocycle} for $U$ is a map 
$c \!: \Gamma \to \mathbb{B}$ that, for all $g_1,g_2 $ in $\Gamma$, satisfies the relation
$$c (g_1 g_2) = c(g_2) + U(g_2) (c(g_1)).$$ 
For each $f \in \Gamma$, we define {\em the drift of $c$ at $f$ as}
$$\mathrm{drift}_c (f) := \lim_{n \to \infty} \frac{\| c(f^n) \|_{\mathbb{B}}}{n}.$$ 
The limit above exists because the sequence of norms $\| c(f^n) \|_{\mathbb{B}}$ is subadditive; indeed,
$$\| c(f^{m+n}) \|_{\mathbb{B}} = \| c(f^n) + U(f^n) (c(f^m)) \|_{\mathbb{B}} \leq \|c(f^n)\|_{\mathbb{B}} + \|U(f^n)(c(f^m))\|_{\mathbb{B}} = \|c(f^n)\|_{\mathbb{B}} + \|c(f^m)\|_{\mathbb{B}}.$$

The next lemma should be compared to \cite{CTV}, and is suitable for applications in wide contexts. 
As the reader will notice, the proof is an adaptation of that of Proposition \ref{prop-two} to this broader context (cf. end of this section). 

\begin{lem} \label{lem-Banach}
Let $U$ be a linear isometric action of a finitely generated Abelian group $\Gamma$ on 
a Banach space $\mathbb{B}$, and let $c \!: \Gamma \to \mathbb{B}$ be a cocycle. Then there exists a 
sequence of vectors $\psi_n \in \mathbb{B}$ such that, for all $f \in \Gamma$, the coboundary defect
$$\big\| c(f) - \big( \psi_n - U(f) (\psi_n) \big) \big\|_{\mathbb{B}}$$ 
converges to $\mathrm{drift}_c (f)$ as $n$ goes to infinity.
\end{lem}

\begin{proof} We number the elements of $\Gamma$ as $f_1,f_2, \ldots$ 
Let us denote 
$$B (n) := \{ f_{1}^{m_1} f_2^{m_2} \cdots f_n^{m_{n}}: \,\, 0 \leq m_i < n \}$$
For each $n \geq 1$, define
\begin{equation}
\psi_n := \frac{1}{n^n} \sum_{g \in B(n)} c(g).
\label{int-zero}
\end{equation}
Each $f \in \Gamma$ equals $f_i$ for a certain index $i$. Then, for each $n \geq i$, 
\begin{eqnarray*}
U(f) (\psi_n) \!\!\!\!\!\!\!
&=& \!\!\!\!\!\!\! \frac{1}{| B(n) |} \sum_{g \in B(n)} U(f) ( c(g) )
\quad = \quad
 \frac{1}{| B(n) |} \sum_{g \in B(n)} [c(g f) - c(f)] \qquad \hspace{1cm} \\
 &\qquad \qquad = \,\,\, & - c(f) + \frac{1}{ | B(n) | } \sum_{g \in B(n)} c(g f) 
\quad = \quad  - c(f) + \frac{1}{ | B(n) | } \sum_{g \in B(n)} c(f g).
\end{eqnarray*}
Therefore, 
$$\Big\| c ( f ) - ( \psi_n - U(f) (\psi_n) ) \Big\|_{\mathbb{B}} \leq \frac{1}{n^n} \Bigl\| \sum_{g \in B(n)} [ c (f g) - c(g) ] \Bigr\|_{\mathbb{B}},$$
and the last expression equals
$$\frac{1}{n^n} \left\| \sum_{\substack{0 \leq m_j < n \\ j \neq i}} \big[ c (f^n f_1^{m_1} \cdots f_{i-1}^{m_{i-1}} f_{i+1}^{m_{i+1}} \cdots f_n^{m_n}) 
- c (f_1^{m_1} \cdots f_{i-1}^{m_{i-1}} f_{i+1}^{m_{i+1}} \cdots f_n^{m_n}) \big] \right\|_{\mathbb{B}} \! \! .$$
By the cocycle relation, this reduces to
$$\frac{1}{n^n} \left\| \sum_{\substack{0 \leq m_j < n \\ j \neq i}} 
U (f_1^{m_1} \cdots f_{i-1}^{m_{i-1}} f_{i+1}^{m_{i+1}} \cdots f_n^{m_n}) (c (f^n) ) \right\|_{\mathbb{B}},$$
which, by the triangular inequality, is smaller than or equal to 
$$\frac{1}{n^n} \sum_{\substack{0 \leq m_j < n \\ j \neq i}} 
\left\| U (f_1^{m_1} \cdots f_{i-1}^{m_{i-1}} f_{i+1}^{m_{i+1}} \cdots f_n^{m_n}) (c (f^n) ) \right\|_{\mathbb{B}} 
= \frac{1}{n^n} \sum_{\substack{0 \leq m_j < n \\ j \neq i}} \| c(f^n) \|_{\mathbb{B}}.$$
The last expression is equal to
$$\frac{1}{n^n} \, n^{n-1} \| c (f^n) \|_{\mathbb{B}} = \frac{\| c (f^n) \|_{\mathbb{B}}}{n}.$$
By definition, this converges to $\mathrm{drift}_c (f)$ as $n \to \infty$. Therefore, the lim sup of
$$\| c ( f ) - ( \psi_n - U( f ) (\psi_n) ) \|_{\mathbb{B}} $$
is at most $\mathrm{drift}_{c} (f)$. 

Conversely, if for $f \in \Gamma$ and $\psi \in \mathbb{B}$ we let 
\, $C := \| c(f) - (\psi - U(f) (\psi)) \|_{\mathbb{B}}$, \, 
then, for each $i \geq 1$, we have
$$C = \| U (f^i) (c(f)) - (U(f^i)(\psi) - U(f^{i+1})(\psi)) \|_{\mathbb{B}}.$$
The triangular inequality and the cocycle relation (together with $c (id) = 0$) then yield
\begin{eqnarray*}
n \, C 
&=& \sum_{i=0}^{n-1}  \big\| U (f^i) (c(f)) - (U(f^i)(\psi) - U(f^{i+1})(\psi)) \big\|_{\mathbb{B}} \\
&\geq& \left\|  \sum_{i=0}^{n-1} U (f^i) (c(f)) -  \big( U(f^i)(\psi) - U(f^{i+1})(\psi) \big) \right\|_{\mathbb{B}} \\
&=& \left\| \sum_{i=0}^{n-1} [c(f^{i+1})-c(f^i)] - \big( \psi - U(f^n) (\psi) \big) \right\|_{\mathbb{B}} \\
&=& \big\| c(f^n) - ( \psi - U(f^n)( \psi )) \big\|_{\mathbb{B}} \\
&\geq& \|c(f^n)\|_{\mathbb{B}} - \|\psi\|_{\mathbb{B}} - \|U(f^n) (\psi) \big) \|_{\mathbb{B}}.
\end{eqnarray*}
Therefore,
$$C \geq \frac{\|c(f^n)\|_{\mathbb{B}}}{n}  - 2 \frac{\|\psi\|_{\mathbb{B}}}{n}.$$
Passing to the limit in $n$ this yields $C \geq \mathrm{drift}_c (f).$
\end{proof}

It follows as a corollary of the proof above that for every cocycle $c$ associated to a linear isometric 
action $U$, the drift of $c$ at $f \in \Gamma$ equals its {\em coboundary defect}, which is defined as 
\begin{equation}\label{eq-drift-cob}
\mathrm{drift}_c (f) = \inf_{\psi \in \mathbb{B} }\big\| c(f) - \big( \psi - U(f) (\psi) \big) \big\|_{\mathbb{B}}.
\end{equation}

If $\Gamma$ is an Abelian group of $C^{1+\mathrm{ac}}$ diffeomorphisms of a compact 1-manifold 
$V$, then $U \!: (f,\varphi) \mapsto ( \varphi \circ f) \cdot Df$ is an isometric action on $\mathbb{B} = L^1(V)$, 
and $c(f) := \frac{D^2 f}{D f}$ is a cocycle for this action. The drift of this cocycle at $f$ is nothing but 
the asymptotic variation of $f$. In this view, equality (\ref{eq-drift-cob}) above should be compared 
to (\ref{eq-conj}). This is another justification for the use of the terminology $C^{1+ac}$-norm for a 
diffeomorphism: roughly, via the asymptotic variation, we transform conjugacy issues into questions 
related to paths in an $L^1$ space. 


\subsection{Mather invariant and the fundamental inequality}
\label{section-mather}

For every $C^2$ diffeomorphism $f$ of $[0,1)$ with no fixed point in the interior, George Szekeres has built in \cite{szekeres} 
a $C^1$ \emph{generating vector field}, that is, a complete vector field for which $f$ is the time-$1$ map of its flow. In the Appendix I
of this work, we carry out a non straightforward extension of this classical construction to $C^{1+\mathrm{bv}}$ diffeomorphisms 
and study its continuity properties. 

\medskip

\noindent\textbf{Warning.} 
In order to reduce the amount of notation, in what follows we will often 
identify
a vector field $\X$ on an interval 
$I$ of $\R$ with the \emph{function} $dx(\X)$, where $x$ denotes the coordinate on $\R$. With this abuse, given a 
diffeomorphism $h$, the pushforward $h_*\X$ will become the function $(Dh \times \X)\circ h^{-1}$, and the pull-back 
$h^*\X$ the function $\frac{\X\circ h}{Dh}$.

\medskip 

With the extension of Szekeres' vector fields  at hand, we can proceed to extend the definition of the \emph{Mather invariant} to 
interval diffeomorphisms of regularity lower than $C^2$ along the classical lines. Namely, we let $\Diff^{2,\Delta}_+ ([0,1])$ (resp. 
$\Diff^{1+\mathrm{bv},\Delta}_+ ([0,1])$, $\Diff^{1+\mathrm{ac}, \Delta}_+([0,1])$) be the set of $C^2$ (resp. $C^{1+\mathrm{bv}}$, 
$C^{1+\mathrm{ac}}$) diffeomorphisms of the interval with no fixed point in the interior.  (The letter $\Delta$ stands for the latter condition.) 
For $f\in \Diff^{1+\mathrm{bv},\Delta}_+ ([0,1])$, let $\X$ and $\Y$ be the left and right vector fields of $f$, respectively. (The former 
arises by seeing $f$ as a diffeomorphism of $[0,1)$, and the latter as a diffeomorphism of $(0,1]$.) Together with them comes a 
{\em Mather diffeomorphism} $M_f :=  P_{\Y} \circ P_{\X}^{-1}$, where $P_{\X}$ (resp. $P_{\Y}$) is the $C^{1+\mathrm{bv}}$ diffeomorphism 
from $(0,1)$ to $\R$ induced by $\X$ (resp. $\Y$) sending some fundamental interval $[a,f(a)]$ of $f$ to $[0,1]$. 
In concrete terms, 
$$P_{\X} = P \!: x\in(0,1)\mapsto \int_a^x\frac{du}{\X(u)},$$
and a similar formula stands for $P_{\Y}$ (with perhaps a different choice for the point $a$).  Since $f$ is the time-1 map of the 
flows of both $\X$ and $\Y$, the map $M_f$ commutes with the integer translations, and it is hence the lift of a circle diffeomorphism. 
The {\em Mather invariant} of $f$ is the class of this circle diffeomorphism (also denoted $M_f$) modulo composition with rotations 
on the left and right. (These naturally come from the choice of the point $a$ in (\ref{eq:P}); see \cite[\S2]{EN} for further details.)  
One says that this invariant is trivial if it coincides with the class of rotations.

Although the Mather invariant is not a genuine circle diffeomorphism (but an equivalence class of them), the total variation of the logarithm 
of its derivative is well defined, since pre/post-composition with rotations does not change its value. The next result that relates this value with 
the asymptotic variation was obtained in \cite{EN} for $C^2$ diffeomorphisms. The proof of this extended version is given in the Appendix I.

\medskip

\begin{thm} \label{thm-general}
For every $f \in \Diff^{1+\mathrm{bv},\Delta}_+ ([0,1])$, one has
$$\big| \var (\log DM_f) - \dist (f) \big| \leq | \log Df(0) | + | \log  Df(1 ) | .$$
\end{thm}

\medskip

The following corollary of the previous theorem will be very useful for us.

\medskip

\begin{cor} \label{cor-general-trivial}
For every $f \in \Diff^{1+\mathrm{bv},\Delta}_+ ([0,1])$ with trivial Mather invariant, one has
$$\dist (f) = | \log D f(0) | + | \log  Df(1) | .$$
\end{cor}

\begin{proof} Triviality of the Mather invariant of $f$ is equivalent to 
\esp $\var (\log DM_f) = 0$. \esp By Theorem \ref{thm-general}, this implies 
$$\dist (f)  \leq | \log  Df(0) | + | \log  Df(1) | .$$
To show the reverse inequality, just notice that, for every $n \geq 1$, 
$$\var (\log Df^n) 
\geq | \log Df^n (1) - \log Df^n (0) | 
= n \big[ | \log Df(1) | + | \log Df (0) | \big].$$
Dividing by $n$ both sides of this inequality and letting $n \to \infty$ yields the desired estimate.
\end{proof}

\medskip

Mather invariant remains unchanged under $C^1$ conjugacy. Besides, together with the conjugacy classes of the germs at the endpoints, 
it totally describes $C^1$ conjugacy classes of diffeomorphisms in $ \Diff^{1+\mathrm{bv},\Delta}_+ ([0,1])$. It is known to be trivial if and only 
if the $C^{1}$ centralizer of the diffeomorphism is isomorphic to $\mathbb{R}$ (and coincides with the flow of the generating vector field). 
Otherwise, this centralizer reduces to $\mathbb{Z}$, and its generator is a root of the diffeomorphism. All these results were established by 
John Mather, and are carefully developed in Chapter V of Yoccoz' thesis \cite{yoccoz}.  (Proofs for $C^2$ diffeomorphisms 
therein extend with no changes to $C^{1+\mathrm{bv}}$ maps once the existence of generating vector fields is established.)


\section{The case of a trivial Mather invariant}
\label{section-trivial}

We consider a nonnecessarily faithful (yet nontrivial) action of $\mathbb{Z}^d$ by $C^{1+\mathrm{bv}}$ diffeomorphisms on the interval $[0,1]$ 
with no global fixed point in the interior. By Kopell's lemma \cite{libro}, every element acting nontrivially admits no fixed point in the interior,  
hence has a well-defined Mather invariant. 

Throughout this section, we assume that an element acting non trivially has a trivial Mather invariant. 
By Mather's theory, if this happens, then it holds for every nontrivial diffeomorphism in the image group. 
We will refer to this setting  just as {\em a $\Z^d$ action with trivial Mather invariant.} 

Assume first that the endpoints are parabolic fixed points for all elements.  (It is easy to see that this holds provided a 
nontrivial element has parabolic endpoints; see \cite[Proposition 8.1]{EN} for a short argument; alternatively, see the discussion below 
on hyperbolic fixed points and centralizers.) In this framework, by Corollary \ref{cor-general-trivial}, the asymptotic variation of every 
element vanishes. Therefore, Proposition \ref{prop-one} implies the following.

\medskip

\begin{lem} \label{lemma-trivial-parabolic}
Consider a $\Z^d$ action by $C^{1+\mathrm{bv}}$ (resp. $C^{1+\mathrm{ac}}$) diffeomorphisms of $[0,1]$ 
with no global fixed point in the interior and trivial Mather invariant. If all elements are parabolic at 
the endpoints, then there exists a $C^{1+\mathrm{bv}}$-continuous (resp. $C^{1+\mathrm{ac}}$-continuous) path of 
simultaneous conjugates $h_t f_i h_t^{-1}$ starting at the original action and finishing at the 
trivial one along which the $C^{1+\mathrm{bv}}$-norms of the generators do not increase.
\end{lem}

\medskip

To deal with hyperbolic fixed points, we first remind some elementary facts about germs of hyperbolic, 1-dimensional linear 
diffeomorphisms. We state them as a remark for future reference.

\begin{rem} \label{rem-affine} 
As it is very well known (see for instance \cite{GT}), the centralizer of a nontrivial linear germ of diffeomorphism 
of the real line fixing the origin coincides with the group of germs of  linear maps fixing the origin. 
Indeed, if $g$ commutes with $x \mapsto \lambda \, x$ then, for every $x \neq 0$, one has, for all $n \in \mathbb{Z}$,
$$g (x) = \frac{g (\lambda^n x)}{\lambda^n} = \left( \frac{g (\lambda^n x) - g(0)}{\lambda^n x - 0} \right) x.$$
Letting $n \to \infty$ or $n \to -\infty$ according to whether $\lambda < 1$ or $\lambda > 1$, respectively, 
the right-side expression converges to $Dg (0) \, x$, which shows that $g$ is linear.

Now, given $\alpha > 0$, we let $h^{\alpha}$ be the germ (at the origin) of  the map $x \mapsto x^{\alpha}$. Notice that this is 
not a $C^1$ diffeomorphism for $\alpha \neq 1$, but the only failure of differentiability arises at the origin (away from it, the 
map is actually a $C^{\infty}$ diffeomorphism). The crucial point that we will exploit is that $h^{\alpha}$ conjugates the linear 
map $x \mapsto \lambda \, x$ to $x \mapsto \lambda^{\alpha} \, x$, which is still a linear map but with a different multiplier.
\end{rem}

 Let us again consider a $\mathbb{Z}^d$-action on $[0,1]$, but this time we assume that an endpoint (say, the origin) is hyperbolic 
for a certain (equivalently, every nontrivial) 
group element $f$. If the action is by $C^2$ diffeomorphisms, then we may use a classical theorem of Shlomo Sternberg 
\cite{sternberg} in its sharp version (due to Yoccoz \cite{yoccoz}): there exists a germ of $C^2$ diffeomorphism $\hat{g}$ such that 
$\hat{g} f \hat{g}^{-1}$ is linear about the origin.\footnote{The Sternberg-Yoccoz theorem still holds for hyperbolic germs of $C^{1+\tau}$ 
diffeomorphisms (see for instance \cite{MW}). The use of this more general version allows extending our Theorem A from $C^2$ to 
$C^{1+\tau}$ commuting diffeomorphisms (with absolutely continuous derivative), with the exact same proof. In particular, 
this applies whenever the affine derivatives lies not only in $L^1$ but also in $L^p$ for some $p>1$.} 
By Remark \ref{rem-affine}, conjugacy by $\hat{g}$ transforms the centralizer of $f$ inside the group of germs (which contains the image 
group of $\Z^d$) into the group of linear transformations. Now, a conjugacy by $h^{\alpha}$ transforms this linear action into another 
action along which the multipliers of group elements at the origin change, and become closer to $1$ as $\alpha$ goes to zero. Finally, 
a conjugacy by $\hat{g}^{-1}$ transforms this new affine action into a new $\Z^d$ action by $C^2$ diffeomorphisms.

We may extend the (local) maps $h^{\alpha}$ and $\hat{g}$ above to $C^2$ diffeomorphisms of $(0,1]$ that coincide with the identity on 
a neighborhood of the right endpoint. We denote $\hat{g}^{\alpha} := (\hat{g}^{-1} h^{\alpha} \hat{g})$, and we state the relevant features of 
this procedure: Conjugacy by $\hat{g}^{\alpha}$ transforms the original $\Z^d$ action into another smooth $\Z^d$ action on $[0,1]$ 
for which the multipliers of group elements at the origin become closer to $1$. Moreover, the map that sends $\alpha$ to the 
$\Z^d$-action conjugated by $\hat{g}^{\alpha}$ is a continuous deformation (starting at $\alpha = 1$) of the original action. 
Furthermore, the Mather invariant of the new action remains trivial. 

The last point deserves some attention. A different view of the previous procedure is that we have changed the Szekeres vector field $\X = \Y$ of 
$f$ to another $C^1$ vector field that coincides with $\alpha \X$ in a neighborhood of the origin and remains untouched near the right endpoint. 
The new action restricted to this neighborhood consists just in integrating this new vector field $\alpha \X$ to the same times of integration of those of 
the original action in regard to $\X$. Since we always keep a $C^1$ vector field defined on the whole interval $[0,1]$, the Mather invariant remains trivial. 

If there is also hyperbolicity at the right endpoint, we simultaneously perform a similar deformation about it (otherwise, we keep 
untouched a neighborhood of this point). We let $g^{\alpha}$ be the resulting conjugating map that involves eventual deformation 
at both endpoints, and we summarize all of this in the lemma below.

\medskip

\begin{lem} Conjugacy by $g^{\alpha}$ yields a continuous deformation of the original action (in the parameter $\alpha \leq 1$)
along which the Mather invariant remains always trivial. Besides, after conjugacy by $g^{\alpha}$, multipliers (at the endpoints) 
change from a value $Df (\cdot)$ to $Df (\cdot)^{\alpha}$.
\end{lem}

\medskip

We are now ready to implement the deformation in the general case of trivial Mather invariant. We stress that the conjugating maps 
$h_t$ we will obtain below do not belong to $\mathrm{Diff}^{1+\mathrm{ac}}_+ ([0,1])$ in case of hyperbolic fixed points, yet they conjugate 
the original action into another $C^{1+ac}$ one.

\medskip

\begin{prop} \label{prop-trivial-M}
For every $\Z^d$ action by $C^2$ diffeomorphisms of $[0,1]$ with no global fixed point in the interior and trivial 
Mather invariant, there exists a $C^{1+\mathrm{ac}}$-continuous path of simultaneous conjugates $h_t f_i h_t^{-1}$ starting at the original 
action and finishing at the trivial one along which the $C^{1+\mathrm{ac}}$-norms of the generators do not increase more than twice the 
$C^{1+\mathrm{ac}}$-norms of the original action.
\end{prop}

\begin{proof} The case of parabolic endpoints is settled by Lemma \ref{lemma-trivial-parabolic}. For the non-parabolic case, 
the heuristic idea is as  follows: we first conjugate the action by the map $g^{\alpha}$ above, where $\alpha < 1$ is close-enough 
to $1$ so that the $C^{1+\mathrm{ac}}$-norms of the generators remain controlled. We next perform the conjugacy procedure of Proposition 
\ref{prop-two}
until we make the total variation of the logarithm of the derivatives of the generators smaller than twice the corresponding 
asymptotic variation, namely
$$2 \, \alpha \cdot \big[ | \log D f_i (0) | + | \log Df_i (1) | \big].$$
Notice that this is a genuine reduction only if $\alpha < 1/2$; however, this is not a major problem.  
(Alternatively, the factor $2$ could be easily replaced by any factor strictly larger than $1$, but we will avoid this argument.) 

Observe that the previous deformation occurs along a $C^{1+\mathrm{ac}}$-continuous path. The idea now is to repeat this argument many times 
so that, in the limit, we obtain the desired path by concatenation. Nevertheless, it is not clear that this process will actually converge (it 
may keep trapped before reaching the end). To overcome this problem we will slightly change our viewpoint using an idea from \cite{Na2}. 
Roughly speaking, instead of directly producing the conjugating maps, we concatenate between their affine derivatives via affine paths: 
this allows keeping a good control on the corresponding $L^1$-norms just by convexity (more precisely, by the triangle inequality).

For concreteness, fix {\em any} $\alpha < 1/2$ and, for each $n \geq 1$, consider the diffeomorphism $g_n$ defined by a formula similar 
to (\ref{eq:def-g_n}) but replacing $f_i$ by $\bar{f}_i := g^{\alpha} f_i (g^{\alpha})^{-1}$. For $N$ large enough, the value of each 
$$\var (\log D (g_N \bar{f}_i g_N^{-1})) = \var (\log D (g_N g^{\alpha} f_i (g_N g^{\alpha})^{-1}) )$$
becomes smaller than or equal to
$$2 \, \dist (\bar{f}_i) = 2 \alpha \, \big[ |\log Df_i (0)| + |\log Df_i (1)| \big] .$$
Set $G_1 := g_N g^{\alpha}$, and consider its affine derivative $\frac{D^2 G_1}{D G_1}$. Although this is not an $L^1$ function, 
the $L^1$-norm of each 
$$\frac{D^2 G_1}{D G_1} \circ f_i \cdot Df_i + \frac{D^2 f_i}{D f_i} - \frac{D^2 G_1}{D G_1}$$
is finite and, actually, smaller than or equal to \, $2 \alpha \, [ | \log Df_i (0) | + | \log Df_i (1) | ]$. Notice that this $L^1$-norm is nothing but 
$$\var (\log D (G_1 f_i G_1^{-1}) ).$$

Repeat this procedure starting with the $\Z^d$ action with generators $G_1 f_i G_1^{-1}$. One thus obtains a map $G_2$ for which 
$$\frac{D^2 G_2}{D G_2} \circ (G_1 f_i G_1^{-1}) \cdot D (G_1 f_i G_1^{-1}) 
+ \frac{D^2 (G_1f_iG_1^{-1})}{D (G_1f_iG_1^{-1})} - \frac{D^2 G_2}{D G_2}$$
is smaller than or equal to 
$$2 \alpha \, \big[ | \log D (G_1 f_i G_1^{-1}) (0) | + | \log D (G_1 f_i G_1^{-1})  (1) | \big]
= 4 \alpha^2 \, \big[ | \log Df_i (0) | + | \log Df_i (1) | \big].$$
Again, this $L^1$-norm is nothing but 
$$\var (\log D ((G_2G_1) f_i (G_2G_1)^{-1}) ).$$

Proceeding this way, we get a sequence of conjugating maps $H_n := G_n \cdots G_2 G_1$ for which 
$$\var (\log D (H_n f_i H_n^{-1})) \leq (2 \alpha)^n \, \big[ | \log Df_i (0) | + | \log Df_i (1) | \big].$$
Now, for $t \in [1-1/n, 1-1/(n+1)]$, let $h_t$ be defined by $h_t (0) = 0$ and 
$$D h_t = C_t \, (D H_n)^{s_t} \, (DH_{n+1})^{1 - s_t},$$ 
where \, $s_t := (1+nt-n) \, (n+1) $ \,
is the affine function in $t$ with value $0$ at $1-1/n$ and $1$ at $1 - 1/(n+1)$,  
and $C_t$ is the unique constant for which
\, $\int_0^1 Dh_t (x) \, dx = 1.$ \,
Then, 
$$D \log D h_t = s_t \log DH_n + (1 - s_t) \log DH_{n+1}.$$
By the cocycle identity of the affine derivative $D^2/D = D (\log D)$, 
this implies that the $L^1$-norm of 
$$\frac{D^2 h_t}{D h_t} \circ f_i \cdot D f_i + \frac{D^2 f_i}{D f_i} - \frac{D^2 h_t}{Dh_t}$$
is smaller than or equal to the sum of the $L^1$-norms of 
$$s_t \left[ \frac{D^2 H_n}{D H_n} \circ f_i \cdot D f_i + \frac{D^2 f_i}{D f_i} - \frac{D^2 H_n}{DH_n} \right]
\quad \mbox{and} \quad 
(1 - s_t) \left[ \frac{D^2 H_{n+1}}{D H_{n+1}} \circ f_i \cdot D f_i + \frac{D^2 f_i}{D f_i} - \frac{D^2 H_{n+1}}{DH_{n+1}} \right].$$
This is bounded from above by
$$s_t \, ( 2 \alpha)^n \, \big[ | \log Df_i (0) | + | \log Df_i (1) | \big] 
+ (1 - s_t) \, ( 2 \alpha)^{n+1} \, \big[ | \log Df_i (0) | + | \log Df_i (1) | \big] ,$$
hence by 
$$( 2 \alpha)^n \, \big[ | \log Df_i (0) | + | \log Df_i (1) | \big].$$
Summarizing, for $t \in [1-1/n, 1-1/(n+1)]$, 
$$\var (\log D (h_tf_ih_t^{-1})) \leq ( 2 \alpha)^n \, \big[ | \log Df_i (0) | + | \log Df_i (1) | \big],$$
and this estimate allows closing the proof.
\end{proof}


\section{The case of a nontrivial Mather invariant}
\label{section-no-trivial}

Again, we consider a non necessarily faithful $\mathbb{Z}^d$ action by $C^{1+\mathrm{ac}}$ diffeomorphisms of 
$[0,1]$ with no global fixed point in the interior, but we now assume that the Mather invariant is nontrivial. 
By Mather's theory, this implies that the image group is isomorphic to $\mathbb{Z}$. We let $f$ be the 
generator of the image group.  As we already mentioned, we would like to deform $f$ and, simultaneously, 
the whole action. However, the $C^{1+\mathrm{bv}}$-norm of $f$ may be very large, and having no control for it 
would lead to loosing any control for the deformation. 

To solve this problem, we apply Proposition \ref{prop-two} to the original action in which we include $f$ as a generator. Notice 
that since the Mather invariant of nontrivial elements is nontrivial, their asymptotic variation is positive. Proposition \ref{prop-two} 
restated as follows will imply that, at the end, the corresponding (conjugate) $f$ attains a small $C^{1+\mathrm{bv}}$-norm.

\medskip

\begin{lem} \label{prop-no-trivial-M}
Let $f_1,\ldots, f_{\ell}$ be $C^{1+\mathrm{bv}}$ commuting diffeomorphisms of $[0,1]$ (not necessarily 
generating a group isomorphic to $\mathbb{Z}^{\ell}$) so that the Mather invariant of the action is not trivial. 
Then there exists a path $(h_t)_{t \in [0,1]}$ of $C^{1+\mathrm{bv}}$ diffeomorphisms of $[0,1]$ starting 
at the identity and such that, for each $1\leq i \leq \ell$, the conjugates $h_t \circ f_i \circ h_t^{-1}$ form a continuous path of 
$C^{1+\mathrm{bv}}$ diffeomorphisms, each of which has $C^{1+\mathrm{bv}}$-norm smaller than or equal to that of 
the corresponding $f_i$, and such that
$$\var \big( \log D (h_1 \circ f_i \circ h_1^{-1}) \big) \leq 2 \, \dist (f_i).$$
If all the $f_i$ are $C^{1+\mathrm{ac}}$, then this deformation occurs along $C^{1+\mathrm{ac}}$ diffeomorphisms, and is continuous for the 
$C^{1+\mathrm{ac}}$-topology.
\end{lem}

\medskip

We are now in position to proceed to the whole deformation. 

\vspace{0.1cm}

\begin{prop} \label{correct}
Assume that a $\mathbb{Z}^d$ action by $C^{1+\mathrm{bv}}$ (resp. $C^{1+\mathrm{ac}}$) diffeomorphisms of $[0,1]$ with no global 
fixed point in the interior has a nontrivial Mather invariant. Then there is a continuous path of actions of $\mathbb{Z}^d$ starting at the 
given one and ending at the trivial action which is continuous with respect to the $C^{1+\mathrm{bv}}$ (resp. $C^{1+\mathrm{ac}}$)-topology. 
Besides, along this path, the $C^{1+\mathrm{bv}}$  (resp. $C^{1+\mathrm{ac}}$)-norm of the generators remains bounded from above by 
twice their $C^{1+\mathrm{bv}}$ (resp. $C^{1+\mathrm{ac}}$)-norm for the original action.
\end{prop}

\begin{proof} We first apply the previous lemma for $\ell := d+1$ letting $f_{d+1} := f$, where $f$ is the generator of the image group. 
The outcome is a path of conjugate actions by $C^{1+\mathrm{bv}}$ (resp. $C^{1+\mathrm{ac}}$) diffeomorphisms $h_t$ along which the $C^{1+\mathrm{ac}}$ 
norms do not increase and, at the end, finishes with an action for which the conjugate $F := h_1 f h_1^{-1}$ of $f$ satisfies 
$$\var (\log DF) \leq 2 \, \dist (f).$$
Now, for each $1 \leq i \leq d$, there exists an integer $m_i$ such that $f_i = f^{m_i} $. 
Using the homogenity of $\dist$ (see (\ref{eq-hom})), for those $i$ for which $m_i \neq 0$, we obtain
\begin{equation}\label{eq:est-F}
\var (\log DF) \leq 2 \, \dist (f) =
\frac{2\, \dist (f^{m_i})}{|m_i|} = \frac{2\, \dist(f_i)}{|m_i|}.
\end{equation} 
Let $F_t$ be the homotopy of $F$ to the identity that is linear on $\log D (\cdot)$.  
More precisely, we let $F_t$ be so that $F_t (0) = 0$ and 
$$DF_t (x) := \frac{ e^{(1-t) \log DF(x)}} {\int_0^1 e^{(1-t) \log DF (y)} dy}$$
We concatenate the path of conjugates by $h_t$ of the given action 
with the path of actions that associate to the $i^{th}$ generator of $\mathbb{Z}^d$ the map $F_t^{m_i}$. 
Since, for a certain constant $c$,  
$$\log DF_t = (1-t) \log DF + c,$$
we have 
\,\, $\var (\log DF_t) = (1-t) \, \var (\log DF).$ \,\,
Therefore, by (\ref{eq:est-F}),
$$\var (\log DF_t^{m_i}) \leq | m_i | \, \var (\log DF_t)
= (1-t) \, | m_i | \, \var (\log DF) \leq 2 \, (1-t) \, \dist (f_i),$$
and this yields the desired path.  
\end{proof}

\begin{rem}\label{salida} 
It is worth pointing out that the deformation technique above applies in all cases where the image group is isomorphic to $\mathbb{Z}$. Thus, 
under this assumption, if the Mather invariant is trivial, both the methods of \S \ref{section-trivial} and \S \ref{section-no-trivial} are suitable 
to deform the action into the trivial one. For technical reasons that will be clarified below, we will prefer the last of these two deformations.
\end{rem}


\section{End of the proof of Theorem \ref{t:A}}
\label{section-todo}

We now proceed to the proof of Theorem \ref{t:A} in the general case.

\subsection{The case of the interval}
Let $f_1,\ldots,f_d$ be commuting $C^2$ diffeomorphisms of $[0,1]$, and denote by $\mathcal{I}$ the 
family of connected components $I$ of the complement of the set of their common fixed points. Since the $C^{1+\mathrm{ac}}$-norm is invariant under 
affine rescaling, to the action restricted to each $\bar I$ we may apply either Proposition \ref{prop-trivial-M} (in case of a trivial 
Mather invariant and higher rank image group), or Proposition \ref{correct} (in case of nontrivial Mather invariant), or Remark \ref{salida} (in case of trivial 
Mather invariant and image group isomorphic to $\mathbb{Z}$). Doing so, we obtain $C^{1+\mathrm{ac}}$-continuous paths of actions on each $\bar I$ 
ending at the trivial action along which the $C^{1+\mathrm{ac}}$ norms of the generators are always bounded from above by \esp $2 \,  \var (\log Df_i |_I ).$ \esp 
Putting all these deformations together, we claim that we obtain a path of $C^{1+ac}$ actions $(F_t)_i$ ending at the trivial action. Notice that in presence of interior hyperbolic fixed points, conjugacies on the left and right keep the action smooth provided we chose along the 
path the same parameter $\alpha$ for the conjugator $x \mapsto x^{\alpha}$ (in case of a higher-rank image group; see \S \ref{section-trivial}) and/or 
the same parameter for the linear homotopy (in case of a image group isomorphic to $\mathbb{Z}$; see \S \ref{section-no-trivial}). 
For this it is worth to remark that if the image group is higher-rank on a fixed interval $I$ with a hyperbolic endpoint, then it is also 
higher-rank at the fixed interval which is on the other side of this point (unless this point is $0$ or $1$).

Continuity of this path is straightforward to prove. Indeed, given $\varepsilon > 0$, we can choose a finite subfamily $\mathcal{J}$ 
of $\mathcal{I}$ such that, for each $i$,  
$$\sum_{I \notin \mathcal{J}} \var (\log Df_i; I) < \frac{\varepsilon}{4}.$$
Let $N$ denote the cardinal of $\mathcal{J}$. 
Given a time $t_0 \in [0,1]$, we can choose $\delta > 0$ so that, for all $|t-t_0| < \delta$, all $I \in \mathcal{J}$ 
and all $i$, the restriction to $I$ of $(F_t)_i$ is $\varepsilon/2N$-close to that of $(F_{t_0})_i$ in the $C^{1+\mathrm{ac}}$ topology. 
Since along the deformation the $C^{1+\mathrm{ac}}$-norm of the $i^{th}$ generator 
restricted to each $I$ remains bounded from above by $2 \, \var (\log Df_i; I)$, this implies that $(F_t)_i$ and $(F_{t_0})_i$ are at a distance at most 
$$2 \sum_{I \notin \mathcal{J}} \mathrm{var} (\log Df_i ; I) + \frac{\varepsilon}{2N} \cdot |\mathcal{J}| < \varepsilon,$$
thus showing continuity at $t_0$.

\subsection{The case of the circle}

Suppose now that $f_1,\ldots,f_d$ are commuting $C^{1+\mathrm{ac}}$ circle diffeomorphisms. If one of them has irrational rotation number, then the 
existence of a $C^{1+\mathrm{ac}}$-continuous path of conjugates $h_t f_i h_t^{-1}$ ending at an action by rotations follows from \cite[Main Theorem]{Na1}, 
which may be seen as an application of Proposition \ref{prop-one}. (Just notice that, although this result is stated for faithful actions in 
\cite{Na1}, it does not use faithfulness along the proof.) The key point here is knowing that the asymptotic variation vanishes, and for  
this it is crucial to assume $C^{1+\mathrm{ac}}$ regularity rather than $C^{1+\mathrm{bv}}$ (see \cite[Theorem 2]{Na1} on this).

\medskip

The case where all the $f_i$ have a rational rotation number is much less trivial than expected. In this situation, the rotation number function 
$\rho$ yields a group homomorphism into $\mathbb{T}^1$ with finite image (see \cite{libro} for this and other structure results used below). 
One is then tempted to apply the arguments of the interval case to the action of $ker (\rho)$, which is the finite-index subgroup formed by 
the elements having fixed points 
(and that, actually, have common fixed points). Nevertheless, this requires several adjustments.

Let $\Gamma$ be the group generated by $f_1,\ldots,f_d$ and $n$ be the cardinality of $\rho(\Gamma)$. 
Let $f_*$ be such that $\rho(f_*) = 1/n$; in particular, $\rho (f_*)$ generates the image group $\rho (\Gamma)$. We assume that $n \geq 2$, as $n=1$ 
is essentially the same as the case of the interval and can be settled in a similar way: one should just take care in preserving the same multiplier at the 
endpoints along the deformation in case this comes from a hyperbolic fixed point in the circle, and this is ensured by the method we have employed. 

Every group element uniquely writes as a product $f_*^{i} \bar{f}$, with $0 \leq i < n$ and $\rho (\bar{f}) = 0$. 
Besides, if $p_0$ denotes a point that is fixed by all elements in $ker (\rho)$ then, 
letting $p_i \!:=\! f_*^i (p_0)$, all the intervals $[p_i, p_{i+1}]$ are fixed by these elements. 

Assume for a while that the group $ker (\rho)$ admits no global  fixed point in the interior of $[p_0,p_1]$. By Kopell's lemma, this 
is the case of every element therein acting nontrivially. 
Then there are two cases to consider.

\medskip

\noindent{\bf The action of $ker (\rho)$ on $[p_0,p_1]$ has a trivial Mather invariant.} If the endpoints are parabolic, then we are in the case 
of vanishing asymptotic variation  for the action on $[p_0,p_1]$ and, by commutativity, on each $[p_i,p_{i+1}]$. This easily implies that the 
asymptotic variation of each circle diffeomorphism $f_j$ vanishes, which allows using Proposition~\ref{prop-one} to obtain a deformation 
of the action on the circle to an action by a finite-order rotation.
If $p_0$ is hyperbolic for a certain element in $ker (\rho)$, then this is also the case for every nontrivial element 
in $ker (\rho)$, and (because of the commutativity with $f_*$) this also holds at every point $p_i$. Conversely, if some $p_i$ 
is hyperbolic, then $p_0$ also is. We then conjugate by maps of type 
$g^{\alpha}$ as before at each of these points in an equivariant way. Notice that this can be done because, by the Sternberg-Yoccoz' 
theorem, there is a (unique) linear coordinate around each of these points, and $f_*$ must conjugate the one at $p_{i}$ into that at 
$p_{i+1}$ by commutativity. In this way, we can build a $C^{1+\mathrm{ac}}$-continuous path of conjugate actions (with conjugating maps 
that are no longer of class $C^{1+\mathrm{ac}}$) along which asymptotic variation becomes smaller and smaller, and the concatenation 
trick of the corresponding affine derivatives allows concluding the proof as in \S \ref{section-trivial}.

\medskip

\noindent{\bf The action of $ker (\rho)$ on $[p_0,p_1]$ has a nontrivial Mather invariant.} In this case, the restriction 
of $ker(\rho)$ to $[p_0,p_1]$ is either trivial, in which case the group is finite, hence conjugate to a group of rotations (this 
conjugacy can be obviously achieved along a path),  or generated by a single map. In the last case, we would like to apply the 
argument of \S \ref{section-no-trivial}. Nevertheless, we cannot proceed so easily since we need to preserve the equivariance 
under the action of $f_*$. To do this, we let $f$ be the diffeomorphism of $[p_0,p_1]$ that generates the action of $ker (\rho)$ on 
$[p_0,p_1]$. Notice that $f_*^n |_{[p_0,p_1]} = f^k$ for a certain integer $k \neq 0$. We start by conjugating the action
so that, in the end, 
all the intervals $[p_i, p_{i+1}]$ have  the same length. Then, we conjugate 
again
so that $f_*$ 
becomes a rotation from $[p_i, p_{i+1}]$ onto $[p_{i+1},p_{i+2}]$ for $0 \leq i < n-1$ 
(this can be easily achieved by means of a cohomological equation on $\log (Df)$). 
Notice that all this procedure can be done through conjugacy by a $C^{1+\mathrm{ac}}$-continuous path of diffeomorphisms 
$h_t$. If we denote $\hat{f}_*$ (resp. $\hat{f}$) the conjugate by $h_1$ of $f_*$ (resp. $f$, where it is defined), 
then $\hat{f}_* |_{[p_{n-1},p_n]}$ becomes a composition of a rotation and $\hat{f}^k$.

Now we apply Proposition \ref{prop-two} in order to conjugate the action of $\ker (\rho)$ on $[p_0,p_1]$ along a path so 
that the $C^{1+\mathrm{ac}}$-norm of $\hat{f}$ becomes very close to $\dist (\hat{f})$. We extend this deformation to the 
remaining intervals via conjugacy by the corresponding rotations. Next, as in \S \ref{section-no-trivial}, we deform the 
conjugate version of $\hat{f}$ along its graph until reaching the identity, and we extend this deformation to the rest of the 
intervals again via conjugacy by rotations. Finally, we extend this deformation to $\hat{f}_*$ so that it coincides with 
a rotation except for the last interval, where it is forced to coincide with the composition of a rotation with the $k^{th}$ power 
of the corresponding deformed version of $\hat{f}$. We leave to the reader to check that this deformation is well behaved 
(with a good control on the $C^{1+\mathrm{ac}}$-norm)
 and ends at
 an action by a single rotation of order $n$.

\medskip

Now we no longer assume that $ker (\rho)$ acts with no global fixed point in the 
interior of $[p_0,p_1]$. In this case, we have a countable family $\mathcal{I}$ of closed intervals $I$ 
in $[p_0,p_1]$ with disjoint interior on each of which $ker (\rho)$ acts with no global fixed point inside. It is not hard to see that we may apply 
the arguments above to the restriction of the action on the union of intervals $I \cup f_* (I) \ldots \cup f_*^{n-1} (I)$. (Indeed, most of our results 
work  verbatim for non-connected compact 1-manifolds...) Pasting together the corresponding deformations along different $I \in \mathcal{I}$ 
yields the desired $C^{1+\mathrm{ac}}$-continuous path. The only delicate issue is, again, that of the multipliers at the hyperbolic periodic 
points, but this is still ensured by the uniqueness of linear coordinates around them (and the fact that, in case of a trivial Mather invariant, 
we prefer to deform along the graph of the generator instead of using the asymptotic variation; see Remark \ref{salida}).

\medskip

In all cases, we have connected the original action with an action by isometries. Since any two actions by rotations can be connected just 
by moving the angles, this allows connecting any two $\Z^d$ actions, thus completing the proof.


\section{
The
proof of Theorem \ref{t:B}}
\label{s:B}

The key ingredient in the proof of Theorem \ref{t:B} is the following result, which corresponds to Proposition 8.4 of \cite{EN}:

\vspace{0.1cm} 

\begin{prop}
\label{p:C1toC2} 
Let $X$ be a $C^1$ vector field on $[0,1]$ with flow $(f^t)_{t\in\R}$. Suppose that the set of times $t$ for which $f^t$ is a $C^2$ 
diffeomorphism is a dense subgroup of $\R$. Then there exists a sequence of $C^2$ diffeomorphisms $h_n$ such that $(h_n)_*X$ converges 
in the $C^1$ sense towards a $C^2$ vector field $\tilde X$ and that, for every $C^2$ diffeomorphism $f^\tau$ of the flow of $X$, the conjugate 
maps $h_n\circ f^\tau\circ h_n^{-1}$ converge in the $C^2$ sense towards the time-$\tau$ map of~$\tilde X$.
\end{prop}

\vspace{0.1cm} 

We may now proceed to the proof of Theorem \ref{t:B}. As we pointed out in \S \ref{th-B}, we restrict to the case $d=2$ just for simplicity. 
Let hence $(f_1,f_2)$ be a pair of commuting $C^2$ diffeomorphisms of $[0,1]$ and let $\eps>0$. 

\vspace{0.1cm}

\begin{lem}\label{lem-AB}
There exists a finite subdivision $0=a_0<\dots<a_n=1$ 
of $\, [0,1]$ such that: 
\begin{itemize}
\item at every $a_i$ in the interior of $[0,1]$, both $f_1$ and $f_2$ are $C^2$-tangent to the identity (in particular, the $a_i$'s are common fixed points of $f_1$ 
and $f_2$);
\item on every $[a_i,a_{i+1}]$, either $\|f_1-\mathrm{id}\|_2 \leq \eps$ and $\|f_2-\mathrm{id}\|_2 \leq \eps$ (such an interval will be said to be of type A), or $f_1$ 
and $f_2$ have no common interior fixed point at which both are $C^2$-tangent to the identity (this will be referred to as an interval of type B).
\end{itemize}
\end{lem}

\begin{proof}
We define the $a_i$'s by induction. Let $a_0=0$. Assume $a_i<1$ has been defined for some $i\in\N$. Then: 
\begin{itemize}
\item if there exists $a\in(a_i,1]$ such that $\|f_1 - \mathrm{id}\|_2 \leq \eps$ and $\|f_2 - \mathrm{id}\|_2 \leq \eps$ on $[a_i,a]$ 
and that $f_1$ and $f_2$ are $C^2$-tangent to the identity at $a$, we let $a_{i+1}$ be the supremum of such points $a$;
\item if not, $a_i$ cannot be accumulated from the right by points at which $f_1$ and $f_2$ are both $C^2$-tangent to 
the identity. We then let $a_{i+1}$ be the smallest such point if there is any, and $1$ otherwise. 
\end{itemize}
We stop once we have reached an $a_i$ equal to $1$. We claim that this happens. Indeed, otherwise, we would get an increasing sequence 
of points $a_i$ bounded from above by $1$. The limit of this sequence would be a fixed point $a$ at which both $f_1$ and $f_2$ are $C^2$ -tangent to the identity. For a large-enough $i$, we would have $\|f_1 - \mathrm{id}\|_2 \leq \eps$ and $\|f_2 - \mathrm{id}\|_2 \leq \eps$ on 
$[a_i,a]$. However, this is incompatible with the definition of $a_{i+1}<a$. 

By construction, the subdivision we obtain satisfies both properties of the lemma.
\end{proof}

\vspace{0.1cm} 

To pursue the proof of Theorem \ref{t:B}, we will show that the restriction of the action to an interval $I$ of type \ref{t:B} lies in the connected 
component of $(\mathrm{id}_I,\mathrm{id}_I)$. 
As a consequence, if we define a new pair $(\tilde f_1,\tilde f_2)$ by $\tilde f_i=\mathrm{id}$ on the intervals of type A and $\tilde f_i=f_i$ on the intervals of type B, we get a $C^2$ action on $[0,1]$ which lies in the connected component of the trivial action and is at distance less than or equal to $\eps$ of the initial action. Since this can be done for any $\eps>0$, we get that the initial action lies in \emph{the closure of} the connected component of the trivial action, but this component is closed, which concludes the proof of Theorem B.

Let us hence consider the restriction of the action to an interval $I$ of type B, that for simplicity we still denote by $(f_1,f_2)$. 

\vspace{0.1cm}

\begin{lem}
An interval $I$ of type B must be of one of the next two categories: 
\begin{itemize}
\item[(B1)] either $f_1$ and $f_2$ are iterates $f^p$ and $f^q$ of the same $C^2$ diffeomorphism $f$ of $I$, 
and this diffeomorphism is $C^2$-tangent to the identity at each endpoint of $I$ that lies in the interior of $[0,1]$;  

\noindent or

\item[ (B2)]  $f_1$ and $f_2$ correspond to rationally independent times, say $1$ 
and $\alpha$ respectively, of the flow $(f^t)_{t\in\R}$ of a $C^1$ vector field $X$ on $I$.
\end{itemize}
\end{lem}

\begin{proof}[Sketch of proof]
This is a direct consequence of classical works of Kopell \cite{kopell} and Szekeres \cite{szekeres} whenever $f_1$ and $f_2$ have no common 
fixed point in the interior of $I$. The general case was treated in \cite{EB} (cf. Proposition 2.7 therein) for $C^\infty$ diffeomorphisms, 
but the arguments are the same here. If $f_1$ or $f_2$ is the identity on $I$, then we are in situation $\textit{(B1)}$. Otherwise, one can 
derive from \cite{kopell,szekeres} that $f_1$ and $f_2$ have exactly the same fixed points. For every connected component $J$ of 
$I\setminus (\Fix(f)\cap\Fix(g))$, there exists $\alpha\in\R$ such that $f_2|_{J}$ is the time-$\alpha$ map of both Szekeres 
vector fields of $f_1 |_{J}$. Now it follows from the non-tangency to the identity at interior fixed points that this $\alpha$ 
does not depend on the component $J$. If it is rational, equal to $p/q$ then, letting $f$ be the time-$1/q$ map of the 
Szekeres vector fields of $f_1$ on each $J$ (which is a $C^2$ diffeomorphism of $I$ since it can be obtained as a 
composition of powers of $f_1$ and $f_2$), we conclude that we are in situation $\textit{(B1)}$. If it is irrational, then 
the left and right Szekeres vector fields of $f$ must coincide on each $J$, and they yield to a $C^1$ vector field of 
$I$ by the work of Yoccoz \cite[chap. V]{yoccoz}. Therefore, we are in case $\textit{(B2)}$. 
\end{proof}

\vspace{0.1cm}

Finally, to close the proof of Theorem \ref{t:B}, let us consider the case of intervals of type (B1) and (B2) separately.

For intervals of type $\textit{(B1)}$, we let $t \mapsto F_t$ be the continuous path of $C^2$ diffeomorphisms from $f$ to $\mathrm{id}$ obtained 
by taking convex combination of the logarithm of the derivatives (as in \S \ref{section-no-trivial}).
Then $t \mapsto (F_t^p,F_t^q)$ yields the desired path from $(f_1,f_2)$ to $(\mathrm{id}, \mathrm{id})$, 
thus showing that $(f_1,f_2)$ belongs to the path-connected component of $(\mathrm{id},\mathrm{id})$. 

On intervals of type $\textit{(B2)}$, Proposition \ref{p:C1toC2} provides a $C^2$ vector field $\tilde X$ whose pair of time-$1$ and time-$\alpha$ maps 
$(\tilde f^1,\tilde f^\alpha)$ belongs both to the closure of the conjugacy class of the restriction of the pair $(f_1,f_2)$ (and thus to its connected component) 
and to the path-connected component of the trivial action (to which it is connected via $t \mapsto (\tilde f^{t},\tilde f^{t\alpha})$). 

This finishes the proof of Theorem \ref{t:B}.


\section{Appendix I: vector fields for $C^{1+\mathrm{bv}}$ diffeomorphisms of the interval}

Here we proceed to give proofs of the results announced in \S \ref{section-mather}, and later proceed to some further developments. 


\subsection{The construction of the vector field}

Recall that $\Diff_+^{1+\mathrm{bv},\Delta} ([0,1])$ (resp. $\Diff_+^{1+\mathrm{ac},\Delta}([0,1])$) 
denotes the space of $C^{1+\mathrm{bv}}$ (resp. $C^{1+\mathrm{ac}}$) diffeomorphisms of $[0,1]$ with no fixed point in the interior. 
For simplicity, whenever it is defined, we will denote the {\em affine derivative} \, $D\log Df$ \, simply by \, 
$Lf$. (For a $C^{1+\mathrm{ac}}$ diffeomorphism, this is an $L^1$ function.)

\vspace{0.2cm}

\begin{prop} \label{p:improved-szek} 
Given $f\in \Diff^{1+\mathrm{bv}, \Delta}_+ ([0,1))$ such that $f(x)>x$ for every $x\in(0,1)$,  
let $\Delta(x) := f(x)-x $, and let
$$c_0 (f) := \begin{cases}
\frac{\log Df (0)}{Df (0) - 1} \quad \mathrm{ if } \quad Df (0) \neq 0,\\
1 \quad \quad \quad  \quad \mathrm{otherwise}.
 \end{cases}$$
For each $n \geq 0$, let $\X_n := c_0(f)(f^n)_* (\Delta) = c_0 (f) \tfrac{\Delta \circ f^{-n} }{ Df^{-n}}$. Then:
\begin{enumerate}

\item the sequence of vector fields $\X_n$ uniformly converges on every compact subset of~$[0,1)$;
\item its limit $\X$ is uniquely integrable and complete, and $f$ is the time-$1$ map of the corresponding flow;
\item the (well-defined) flow $(f^t)_{t \in \mathbb{R}}$ of $\X $ coincides with the $C^1$ centralizer of $f$;
\item for every $c > 0$, the function $\log \X_n$ converges to $\log \X$ in the BV topology 
on the fundamental interval $[c,f(c)]$, and
$$\left|\var(\log \X;[c,f(c)])- \log Df(0)\right|\le \var(\log Df ; [0,c]).$$
\end{enumerate}
Furthermore, if $f\in \Diff^{1+\mathrm{ac},\Delta}_+ ([0,1))$, then
\begin{enumerate}
\item[5.] 
the function $\log \X$ is absolutely continuous on every fundamental interval $[c,f(c)]$, and
$$\left\| D\log \X-\tfrac{\log Df(0)}\X \right\|_{L^1([c,f(c)])}\le\| L f \|_{L^1([0,c])}.$$
\end{enumerate}
 \end{prop}
 
 \begin{proof}[Proof.] We proceed step by step.

\medskip

\noindent{\em 1.} For every $k\in\N$, we have $ \log \frac{ \X_{k+1} }{ \X_k }  = \theta \circ f^{-k} $ with, for every $x\in(0,1)$,
  $$\theta(x)
  := \log \left( \frac{ x - f^{-1}(x) }{ Df^{-1}(x) \, (f(x) - x) }\right) = \log \left(\int_0^1 Df^{-1}\bigl( x + s (f(x)-x) \bigr)\,ds \right) 
- \log(Df^{-1}(x))$$
(the last equality follows from Taylor's integral formula). In particular, $\theta$ extends to $[0,1)$ as a continuous map, with $\theta(0)=0$. 
By the mean value theorem, for every $x\in[0,1)$,  
$$ \int_0^1 Df^{-1} \bigl( x + s (f(x)-x) \bigr) \, ds = Df^{-1}(y_x) \quad \text{for some} \quad y_x \in [x,f(x)]. $$
Therefore, given $c\in[0,1)$, for every $x\in [0,c]$ and every $0\le i\le j$, we have
\begin{align*} \label{e:somme}
\sum_{k=i}^{j-1} \left| \theta \circ f^{-k}(x) \right| 
&=  \sum_{k=i}^{j-1} |\log Df^{-1}(y_{f^{-k}(x)})-\log Df^{-1}(f^{-k}(x))|\\
&\le \var(\log Df^{-1} ; [0,f^{-i+1}(c)])=\var(\log Df ; [0,f^{-i}(c)])\xrightarrow[i\to+\infty]{}0.
\end{align*}
As a consequence, $\sum_k \theta\circ f^{-k}$ converges absolutely and uniformly on 
$[0,c]$. Denote by $\Sigma$ this sum (which is thus continuous on $[0,1)$). Then $(\X_n)$ 
converges uniformly on every compact subset of $[0,1)$ towards $\X:= c_0(f)\Delta \, e^\Sigma$. 
In particular, $\X$ vanishes only at $0$, and is strictly positive at all points of $(0,1)$.

\medskip

 \noindent{\em 2.} Since by definition $\X_{n+1}=f_*\X_n$, in the limit we obtain $\X=f_*\X$, or equivalently $\X=f^*\X$, 
 that is, $\X=\frac{\X\circ f}{Df}$. This implies that the derivative of $x\mapsto \int_x^{f(x)}\frac{du}{\X(u)}$ is identically 
 $0$ on $(0,1)$, so this map is constant, say equal to some $\tau\in\R^*$. The proof of the equality $\tau=1$ in the proof 
 of the ``usual'' Szekeres theorem ({\em i.e.} when $f$ is assumed $C^2$) reproduced in
 \cite[Proposition 4.1.14]{libro} works without any change in the present setting. 
 
Let us now show that $\X\res{(0,1)}$ is complete and uniquely integrable. Fix $a\in (0,1)$, and let 
\begin{equation}\label{eq:P}
P_{\X} = P \!: x\in(0,1)\mapsto \int_a^x\frac{du}{\X(u)}.
\end{equation}
The map $P$ is of class $C^1$, with positive derivative, and $P(f^n(a))=n$ 
for every $n\in \mathbb{Z}$. Therefore, $P$ defines a $C^1$ diffeomorphism between $(0,1)$ and $\R$. 
Its inverse $\psi := P^{-1}$ hence sends $\R$ into $(0,1)$.
 
 We claim that, for every $(t_0,x_0)\in \R\times (0,1)$, the Cauchy problem 
 $$\begin{cases}
 \dot x = \X(x)\\
x(t_0)=x_0
 \end{cases}$$
 has a unique maximal solution, defined on all of $\R$ by $\gamma(t) := P^{-1}(t-t_0+P(x_0))$. 
 (In particular, the only solution of $\dot x = \X(x)$ passing through $0$ is constant equal to $0$.)
Indeed, given $(t_0,x_0)\in \R\times (0,1)$, one immediately checks that the map $\gamma$ above is a (maximal) solution to the 
above Cauchy problem. To prove uniqueness, first notice that one cannot invoke the Cauchy-Lipschitz theorem, since $\X$ is not 
necessarily locally Lipschitz. However, the one-dimensional setting provides a more elementary argument, as shown below. 
 
 Assume $\tgamma:I\to[0,1)$ is a maximal solution of the same Cauchy problem, and let $I_*\subset I$ be the maximal interval containing $t_0$ on which 
 $\tgamma$ does not vanish. 
 Then for every $t\in I_*$, we have $\dot{\tgamma}(t)=\X(\tgamma(t))$, with $\X(\tgamma(t))\neq0$, hence $\frac{\dot{\tgamma}(t)}{\X(\tgamma(t))}=1$.
 By integration and change of variables, we obtain 
 $$\int_{x_0}^{\tgamma(t)}\frac{du}{\X(u)}=t-t_0,$$ 
 that is, \, $P(\tgamma(t))-P(x_0)=t-t_0,$ \, 
 and so $\tgamma(t)=\gamma(t)$. It hence remains to justify that $I_*=\R$. Assume by contradiction that one of the extrema of $I_*$, 
 say its infimum to fix ideas, is finite. Then the restriction of $\tgamma$ to $I_*$ is not maximal, since it can be extended until $-\infty$ 
 by $\gamma$. Therefore, the infimum $\alpha$ of $I_*$ is not that of $I$. This implies that $\lim_{t \to \alpha} \tgamma (t) = 
 \lim_{t \to \alpha} \gamma (t) = 0$. However, this is in contradiction to $\lim_{t \to \alpha} \gamma (t) = 
 \gamma(\alpha)\in(0,1)$. A similar argument shows that $\sup I_*=+\infty$, thus closing the proof of the uniqueness.  

Summarizing, the flow $(t,x)\mapsto \phi(t,x)$ of $\X$ is well-defined on $\R\times[0,1)$ and given 
by $\phi(t,x)=P^{-1}(t+P(x))$ for $x\neq0$ and $\phi(t,0)=0$ for every $t$. In particular, the equality 
$$\int_x^{f(x)}\frac{du}{\X(u)}=1,$$ 
which is equivalent to $P(f(x))-P(x)=1$, means that $f$ is the time-$1$ map of $\X$. \medskip 

\medskip

\noindent{\em3.} For every $t\in\R$, the time-$t$ map $f^t = \phi(t,\cdot )$ of the flow of $\X$ commutes with $f$. 
Therefore, in order to derive {\em 3.} from Kopell's Lemma \cite{kopell}, 
it is enough to prove that $f^t$ is a $C^1$ diffeomorphism of $[0,1)$ for each $t$. 
The formula for $\phi(t,x)$ given in the proof of {\em 2.} shows that $f^t$ is $C^1$ and $Df^t = \frac{\X\circ f^t}{\X}$ on $(0,1)$. 
We are thus reduced to proving that $Df^t(x)$, or rather \, $\log Df^t(x)$, \, has a limit when $x$ goes to $0$, namely 
\, $t\log Df(0)$. \, To do this, it is enough to restrict to \, $t\in[0,1]$. \, For every $x>0$,
$$\log Df^t(x) = \log \left( \frac{\X(f^t(x))}{\X(x)} \right) = \log \left( \frac{\Delta(f^t(x))}{\Delta(x)} \right) + \Sigma(f^t(x)) - \Sigma(x).$$
Since $\Sigma$ is continuous at $0$ and vanishes at this point, what we need to prove is that 
$$\lim_{x\to0} \log \left( \frac{\Delta(f^t(x))}{\Delta(x)} \right) = t\log Df(0).$$
If $Df(0)=1$ then, for some $u \in [x, f^t (x)]$,
\begin{align*}\left|\frac{\Delta(f^t(x))}{\Delta(x)}-1\right|&=\frac{|\Delta'(u)|\times|f^t(x)-x|}
{\Delta(x)}
\le \max_{y \in [x,f^t(x)]}|\Delta' (y)|\xrightarrow[x\to0]{}|\Delta'(0)|=0.
\end{align*}
If $\lambda:=Df(0)>1$, then $\Delta'(0)=\lambda-1>0$, so 
$\frac{\Delta(y)}{y}\xrightarrow[y\to0]{}\lambda-1\neq0$. This implies $\frac{\Delta(f^t(x))}{\Delta(x)}\sim_{x\to0} \frac{f^t(x)}{x}$.  
Thus, we need to prove that $\log ( \frac{f^t(x)}{x} ) \xrightarrow[x\to0]{}t\log \lambda$. To do this, first observe that 
$$\X(x)=\frac{\log\lambda}{\lambda-1} \, \Delta(x) \, e^{\Sigma(x)}\underset{x\to0}{\sim}(\log\lambda) \, x.$$ 
Thus, given $\eps>0$, we may let $\delta>0$ be such that 
$$\frac{1-\eps}{(\log\lambda)u}<\frac1{\X(u)}<\frac{1+\eps}{(\log\lambda)u}$$
for every $u\in(0,\delta]$. Assuming that $f^t(x)$ (and thus $x$) is in this interval, we obtain
$$\int_x^{f^t(x)}\frac{1-\eps}{(\log\lambda)u}du
\le \int_x^{f^t(x)}\frac{du}{\X(u)}\le  \int_x^{f^t(x)}\frac{1+\eps}{(\log\lambda)u}du,$$ 
hence
$$\frac{1-\eps}{\log\lambda}\log \left( \frac{f^t(x)}{x} \right) \le t \le \frac{1+\eps}{\log\lambda}\log \left(\frac{f^t(x)}{x} \right),$$
and thus
$$\frac{t\log\lambda}{1+\eps}\le \log \left( \frac{f^t(x)}{x} \right) \le \frac{t\log\lambda}{1-\eps}.$$
Letting $\varepsilon \to 0$, this gives the desired limit for \, $\log (\frac{f^t(x)}{x})$.

\medskip

\noindent{\em 4.} Since $\X$ is bounded away from zero on $[c,f(c)]$, by item {\em 1.} we have that $\log \X_k$ 
converges uniformly towards $\log \X$ on this segment. Now for every $k\ge j\in\N$,
\begin{align}\label{eq:est-n}
\var \big( \log \X_k - \log \X_j ; [c,f(c)] \big) \le \var \big(&\log (\Delta\circ f^{-k})-\log (\Delta\circ f^{-j}) ; [c,f(c)] \big)\\&+ 
\var \big( \log Df^{-k}-\log Df^{-j} ; [c,f(c)] \big).\notag
\end{align}
Concerning the last term, 
\begin{eqnarray*}
\var(\log Df^{-k}-\log Df^{-j} ; [c,f(c)]) 
&=& \var \Big( \sum_{i=j}^{k-1}\log (Df^{-1}\circ f^{-i}); [c,f(c)] \Big)\\
&\le& \sum_{i=j}^{k-1} \var( \log (Df^{-1}\circ f^{-i}); [c,f(c)])\\
&=& \sum_{i=j}^{k-1} \var( \log Df^{-1}; [f^{-i}(c),f^{-i+1}(c)])\\
&\le& \var( \log Df^{-1}; [f^{-k+1}(c),f^{-j+1}(c)])\\
&=& \var( \log Df; [f^{-k}(c),f^{-j}(c)])\xrightarrow[j\to+\infty]{}0.
\end{eqnarray*}
Concerning the previous term in (\ref{eq:est-n}), 
\begin{align*}
\var(\log (\Delta\circ f^{-k})-\log (\Delta\circ f^{-j}) ; [c,f(c)])
&=\| D\log (\Delta\circ f^{-k})-D\log (\Delta\circ f^{-j})\|_{L^1([c,f(c)])}\\
&=c_0(f) \left\| \tfrac{D\Delta\circ f^{-k}}{\X_k}-\tfrac{D\Delta\circ f^{-j}}{\X_j}\right\|_{L^1([c,f(c)])}\xrightarrow[j\to+\infty]{}0,
\end{align*}
since $\tfrac{D\Delta\circ f^{-k}}{\X_k}$ converges uniformly towards $\tfrac{Df(0)-1}{\X}$ on $[c,f(c)]$.
By completeness of $BV([c,f(c)])$, we get that \, $\log \X$ \, belongs to this space. Observe furthermore that 
\begin{align*}
\var \big( \log (\Delta\circ f^{-k}) ; [c,f(c)] \big) 
= c_0(f) \int_c^{f(c)}\left|\tfrac{D\Delta\circ f^{-k}}{\X_k}\right| \xrightarrow[k\to+\infty]{}c_0(f)\int_c^{f(c)}\frac{D\Delta(0)}{\X}=\log Df(0),
\end{align*}
and that (because of the previous estimate with $j=0$)
\begin{align*}
\var(\log Df^{-k}; [c,f(c)])\le  \var( \log Df; [f^{-k}(c),c]).
\end{align*}
Therefore, letting $k$ go to infinity in 
$$\var(\log \X_k;[c,f(c)])\le \var \log (\Delta\circ f^{-k}) ; [c,f(c)])+ 
\var(\log Df^{-k}; [c,f(c)])$$
and 
$$\var (\log (\Delta\circ f^{-k}) ; [c,f(c)])\le\var(\log \X_k;[c,f(c)])+ 
\var(\log Df^{-k}; [c,f(c)])$$
(which both follow from the definition of $\X_k$),
we get
\begin{equation*}\label{est-one-direction}
\var(\log \X;[c,f(c)])\le \log Df(0)+\var(\log Df ; [0,c])
\end{equation*}
and
\begin{equation*}\label{est-second-direction}
 \log Df(0)\le \var(\log \X;[c,f(c)])+\var(\log Df ; [0,c]),
\end{equation*}
which yield the desired estimate.
\medskip

\noindent{\em 5.} We now assume $f\in\Diff^{1+\mathrm{ac},\Delta}_+([0,1))$. Then, by the definition, $\log \X_k$ is locally absolutely continuous on $(0,1)$. Since, for 
every $c > 0$, the subspace $AC([c,f(c)])$ is closed in $BV([c,f(c)])$, the function $\log \X$ is also absolutely continuous on $[c,f(c)]$, and so is $\X$. 
One can then improve the estimate of item {\em 4.} as follows: Given $k \geq 0$, the functions $\X_k$ and $Df^k$ are almost everywhere differentiable, 
and the following equalities hold almost everywhere on $[0,1)$:
\begin{small}
$$ D\log \X_k  
= D \left(\log(\Delta \!\circ \!f^{-k}) \! - \! \log(Df^{-k})
\right)
= \tfrac{D \Delta (f^{-k})}{\Delta (f^{-k})} \cdot Df^{-k}(x) - D \log Df^{-k}
= c_0(f)\tfrac{D\Delta\circ f^{-k}}{\X_k} - Lf^{-k}.$$
\end{small}Therefore, 
\begin{eqnarray*}
\left\|D\log \X_k-c_0(f)\tfrac{D\Delta\circ f^{-k}}{\X_k}\right\|_{L^1([c,f(c)])} 
&=& \|Lf^{-k}\|_{L^1([c,f(c)])} \\
&=& \left\| \sum_{i=0}^{k-1} L (f^{-1}) \circ f^{-i} \cdot Df^{-i} \right\|_{L^1 ([c,f(c)])} \\
&\leq& \sum_{i=0}^{k-1} \left\| L (f^{-1}) \circ f^{-i} \cdot Df^{-i} \right\|_{L^1 ([c,f(c)])} \\
&=& \sum_{i=0}^{k-1} \| L(f^{-1}) \|_{L^{1}([f^{-i}(c),f^{-i+1}(c)])}\\
&=& \sum_{i=0}^{k-1} \| L(f) \|_{[f^{-i-1}(c), f^{-i}(c)]} 
\,\,\, \leq \,\,\,  \|Lf\|_{[0,c]},
\end{eqnarray*}
and taking the limit when $k$ goes to infinity gives the desired estimate.\end{proof}


\subsection{Mather invariant and the fundamental inequality revisited} 

Our goal here is to prove Theorem \ref{thm-general}. Namely, for every $f \in \Diff^{1+\mathrm{bv},\Delta}_+([0,1])$, 
\begin{equation}\label{fund-ineq}
\big| \mathrm{var} (\log DM_f) - \dist (f) \big| \leq |\log Df(0)| + |\log Df(1)|.
\end{equation}
As mentioned in \S \ref{section-mather}, this corresponds to an extension of 
\cite[Theorem B]{EN} to the $C^{1+\mathrm{bv}}$ setting.

Assume that $f (x) > x$ for $x \in (0,1)$ to fix ideas (otherwise, just use (\ref{eq-inv}) and the fact that, by definition, 
the Mather invariant of $f^{-1}$ equals that of $f$ up to a reflexion.)  Since $f$ is the time-1 map of the flow of both 
$\X$ and $\Y$, the maps $\psi_{\X} = P_{\X}^{-1}$ and $\psi_{\Y} = P_{\Y}^{-1}$ satisfy $\psi \circ T = f \circ \psi$ for $T := T_1$, 
the translation by $1$. Therefore, for each positive $m,n$ we have, letting $k:= m+n$: 
\begin{equation}\label{ren}
M_f = T_{-m} \circ ( \psi_{\Y} )^{-1} \circ f^k \circ \psi_{\X} \circ T_{-n}.
\end{equation}
This yields,
\begin{eqnarray*}
DM_f (t) 
&=& \frac{D \psi_{\X} (t-n)}{D \psi_{\Y} \big( (\psi_{\Y})^{-1} f^k \, \psi_{\X} (t-n) \big) } \cdot Df^k (\psi_{\X} (t-n)) \\
&=& \frac{\X (\psi_{\X} (t-n)) }{\Y (f^k\psi_\X(t-n))} \cdot Df^k (\psi_{\X} (t-n)).
\end{eqnarray*}
This easily implies that 
$$\left| \var ( \log DM_f  ) - \var (\log Df^k; [f^{-n}(a), f^{-n+1}(a)])\right|$$
is bounded from above by 
$$\var \big( \log \X ; [f^{-n}(a), f^{-n+1}(a)] \big) + \var \big( \log \Y ; [f^m(a),f^{m+1}(a)] \big).$$
By item {\em 4.} of Proposition \ref{p:improved-szek}, the latter expression is smaller than or equal to
$$\big| \log Df(0) \big| + \big| \log Df (1) \big| + \mathrm{var} (\log Df; [f^{-n}(a), f^{-n+1}(a)]) + \mathrm{var} (\log Df ; [f^m (a), f^{m+1}(a)]).$$
Letting $m=n = N \to \infty$, the last two terms above converge to $0$, and Proposition 5.1 of \cite{EN} yields 
\begin{equation}\label{ren-dist}
\var (\log Df^k; [f^{-n}(a), f^{-n+1}(a)]) = \var (\log Df^{2N} ; [f^{-N}(a), f^{-N+1}(a)] )\to \dist(f).
\end{equation}
Putting everything together, we finally obtain 
$$\big| \var ( \log  DM_f ) - \dist (f) \big| \leq \big| \log Df(0) \big| + \big| \log Df (1) \big|.$$

\begin{rem} The proof above has the disadvantage of using a result from \cite{EN}. Notice also that the first proof of \eqref{fund-ineq} given in \cite{EN} 
for $C^2$ diffeomorphisms applies with some minor adjustments to $C^{1+\mathrm{ac}}$ diffeomorphisms, but fails in $C^{1+\mathrm{bv}}$ regularity. 
Below we propose a direct argument at least for half of the inequality. 

From $M_f = ( \psi_{\Y} )^{-1} \circ \psi_{\X}$ we obtain 
$$\log DM_f = (\log \X - \log \Y)\circ \psi_{\X}.$$ 
It readily follows from item {\em 4.} in Proposition \ref{p:improved-szek} (and its analog for $\Y$) that $M_f$ 
belongs to $\Diff^{1+\bullet}_+ (\R/\Z)$ 
for $f\in \Diff^{1+\bullet,\Delta}_+ ([0,1])$, where $\bullet$ stands for either $\mathrm{bv}$ or $\mathrm{ac}$; moreover, 
\begin{align*}
\mathrm{var} (\log DM_f;[0,1]) &\leq \var(\log \X;\psi_{\X}([0,1])+\var(\log \Y;\psi_{\X}([0,1])\\
&=\var(\log \X ;[a,f(a)])+\var(\log \Y ;[a,f(a)])\\
&\le |\log Df(0)|+\var (\log Df ; [0,a])+ |\log Df(1)|+\var (\log Df ; [a,1]).
\end{align*}
Therefore, 
$$\mathrm{var} (\log DM_f)\le  |\log Df(0)|+ |\log Df(1)|+\var (\log Df).$$
Now, since $M_f$ is invariant under conjugacy, this implies
$$\mathrm{var} (\log DM_f) \leq  |\log Df(0)| + |\log  Df(1)| + \inf_{h} \mathrm{var} (\log D (hfh^{-1})),$$
where the infimum runs over all $h \in \mathrm{Diff}^{1+\mathrm{bv}}_+([0,1])$. By (\ref{eq-conj}), this infimum 
is nothing but the asymptotic variation of $f$. Thus, the previous inequality becomes
$$\mathrm{var} (\log DM_f) \leq  |\log Df (0)| + |\log Df (1)| + \dist (f),$$ 
which is one of the inequalities involved in (\ref{fund-ineq}).
\end{rem}


\subsection{The case of piecewise smooth homeomorphisms}

Equality (\ref{ren}) allows thinking of the Mather 
invariant as a renormalization of the action of high powers of $f$. For concreteness, assume again that $f(x) > x$ for 
all $x \in (0,1)$, and suppose that both $m,n$ are positive. Equality (\ref{ren}) then says that, in order to compute 
$M_f$ on $[0,1]$ (which is identified to $[a,f(a)]$ via $\psi_{\X}$) we may proceed by going to the translated point 
$t-n$, look for the image under $k$ iterates of $f$ of $\psi_{\X} (t-n)$, coming back to the real line by $\psi_{\Y}^{-1}$, and 
finally translating by $-m$. This is nothing but looking at the action of $f^k$ from the interval $[f^{-n} (a), f^{-n+1} (a)]$ 
into $[f^m(a),f^{m+1}(a)]$, both identified to the unit segment, the former via $\psi_{\X}^{-1}$ and the latter via $\psi_{\Y}^{-1}$.

There are two applications of this view. The first is that, if we know {\em a priori} the vector fields $\X$ and $\Y$ in neighborhoods 
of the corresponding endpoints, then (\ref{ren}) explicitly gives the Mather invariant. This is particularly useful in the case where 
$f$ is of class $C^2$ and the endpoints are hyperbolic fixed points of $f$. Indeed, in this situation, the Sternberg-Yoccoz 
linearization theorem establishes that the germs of $f$ at these points are $C^2$ linearizable. Therefore, up to a $C^2$ change 
of coordinates, we may assume that $\X (y) = \lambda \, y$ (resp. $\Y (z) = \mu \, (1-z)$) in a neighborhood of $0$ (resp. $1$), where 
$\lambda := \log (Df (0)) > 0$ (resp. $\mu := \log (Df (1)) < 0$). Taking $m,n$ large enough so that the intervals $[f^{-n}(a), f^{-n+1}(a)]$ 
and $[f^m(a), f^{m+1}(a)]$ lie inside the interior of  these domains of linearization, this yields a particularly simple expression for (\ref{ren}).
 
Another application of this view of the Mather invariant is the extension of its definition to homeomorphisms that are $C^{1+\mathrm{bv}}$ except 
for finitely many points in the interior and have nonvanishing left and right derivatives ({\em piecewise $C^{1+\mathrm{bv}}$ diffeomorphisms}, for short). 
Of course, one way to proceed in this case is to allow the vector fields $\X$ and $\Y$ to have discontinuities. Indeed, both $\X$ and $\Y$ are 
well defined in neighborhoods of the corresponding endpoints (because vector fields exist for germs of diffeomorphisms, as easily follows 
from Proposition \ref{p:improved-szek}), and starting from there they can be extended to the whole interval in a unique way by using the equivariance relations
$$\X (f(x)) = \X (x) \cdot Df(x), \qquad \Y (f(x)) = \Y (x) \cdot Df(x).$$
However, equality (\ref{ren}) is in many cases easier to handle. In particular, it leads to the fundamental inequality (\ref{fund-ineq}) in this broader context, 
the proof of which follows along the same lines of the one given above. Notice that both the variation of the logarithm of the derivative and the 
asymptotic variation are well defined for piecewise $C^{1+\mathrm{bv}}$ diffeomorphisms (at break points, we keep the value of the right derivative).

A particularly relevant example of the previous discussion 
is the space $\mathrm{PL}_+^{\Delta}([0,1])$ of piecewise-affine homeomorphisms of the interval with no fixed point in the 
interior. In this context, the variation of the Mather invariant described above has been considered by many authors.  A nice 
review of all of this may be found in \cite{matucci}. In particular, one can find therein a proof of the fact that $M_f$, together 
with the multipliers $Df (0)$ and $Df(1)$, are a complete invariant of $PL_+$ conjugacy in $\mathrm{PL}_+^{\Delta}([0,1])$. 
(These are analogous results to those of Mather that hold for $C^2$ diffeomorphisms.)

There are several other special features of piecewise-affine homeomorphisms in this context. One is that the conjugating maps that realize the 
asymptotic variation as the infimum of the total variation of the logarithm of the derivative along the conjugacy class may be also taken to be 
piecewise-affine. This immediately follows from the  explicit formula (\ref{eq:def-g_n}) that defines them. In this regard, it would be interesting 
to further study the case of piecewise-projective homeomorphisms: can the conjugating maps be taken also being piecewise-projective ?

Another special feature concerns equality (\ref{ren-dist}), namely
$$\lim_{N \to \infty} \var \big( \log Df^{2N} ; [f^{-N}(a), f^{-N+1}(a)] \big) = \dist (f).$$
Indeed, for a large-enough $N$, the left-hand side expression above obviously stabilizes. More generally, 
let $f$ be piecewise $C^{1+\mathrm{bv}}$ so that it is affine on neighborhoods of  both $[0,\varepsilon]$ and $[1-\varepsilon,1]$ 
for a certain $\varepsilon > 0$. 
Let $a$ be a point in the interior of one of these intervals such that $f(x)$ also lies therein, and let  $k$ be a positive 
integer such that either $f^k (a) > 1-\varepsilon$ or $f^{k} (a) < \varepsilon$, according to whether $f$ moves interior points 
to the right or to the left. If we denote by $I$ the interval with endpoints $a,f(a)$, then 
\begin{equation}\label{eq:K}
\dist (f) = \var (\log Df^k; I).
\end{equation}

A third special feature concerns the fundamental inequality (\ref{fund-ineq}), which in this case becomes an exact equality. More 
precisely, remind that for $C^2$ (and, more generally, for $C^{1+\mathrm{ac}}$) diffeomorphisms, one always has the strict inequality
$$\var(\log DM_f ) < \big| \log Df(0) \big| + \big| \log Df (1) \big| + \dist (f)$$
whenever the endpoints are hyperbolic fixed points; see \cite[Proposition 4.4]{EN}.  
However, for piecewise-affine homeomorphisms, the left and right-hand-side expressions above become equal.

\medskip

\begin{prop}  For every $f \in \mathrm{PL}_+([0,1])$ one has
\begin{equation}\label{fund-eq}
\mathrm{var} (\log DM_f ) = \big| \log Df(0) \big| + \big| \log Df (1) \big| + \dist (f).
\end{equation}
\end{prop}

\begin{proof} We assume that $f(x) > x$ for $x \in (0,1)$, and we use again (\ref{ren}). By differentiation, this becomes
$$DM_f (t) = \frac{\X (\psi_{\X} (t-n)) }{\Y (f^k\psi_\X(t-n)) \cdot Df^k (\psi_{\X} (t-n))}
$$ 
(at discontinuity points of the derivative, this equality holds for left and right derivatives). Let $p_0 < p_1 < \ldots < p_n$ 
be a finite family of points of $[a,f(a)]$ that includes all discontinuity points of $Df^k$ therein, as well as $a$ and $f(a)$. 
Slightly changing $a$ if necessary, we may ensure that neither $a$ nor $f(a)$ are among these discontinuity points. Then
$$\mathrm{var} (\log DM_f) 
= \sum_{i=1}^{n-1} \mathrm{var} \Big( \log \big( \tfrac{\X}{\Y \circ f^k} \big) + \log Df^k ; [p_{i-1},p_i) \Big).$$
Notice that $\log (Df^k)$ has no variation on $[p_{i-1},p_i[$, since its variation may only arise as Dirac jumps at  the points $p_i$. 
Moreover, such a point adds $| \log Df^k_+ (p_i) - \log Df^k (p_i) |$ to the variation of $\log DM_f$. All these terms 
add up to $\dist (f)$, because of (\ref{eq:K}). Hence,
\begin{eqnarray*}
\mathrm{var} (\log DM_f) 
&=& \sum_{i=1}^{n-1} \mathrm{var} \Big( \log \big(  \tfrac{\X}{\Y \circ f^k}\big) \Big) + \dist (f) \\
&=& \mathrm{var} \Big( \log \big(  \tfrac{\X}{\Y \circ f^k} \big) ; [a,f(a)]\Big) + \dist (f).
\end{eqnarray*}
Finally, on the interval 
$[a,f(a)]$, the function $\X$ is strictly increasing (equal to $\lambda x$, with $\lambda = \log Df (0) > 0$), 
while $\mathcal{Y} \circ f^k$ is strictly decreasing (equal to $ \mu (1 - f^k (x))$, with $\mu = \log Df (1) < 0$). This yields
\begin{small}\begin{eqnarray*}
\mathrm{var} \Big( \log \big( \tfrac{\X}{\Y \circ f^k} \big) ; [a,f(a)]\Big) 
&=& \big| \log ( \tfrac{\X}{\Y \circ f^k})(f(a)) - \log ( \tfrac{\X}{\Y \circ f^k}) (a) \big| \\
&=& \big| \log (e^{\lambda} \lambda a) - \log (\lambda a) + \log (\mu (1-f^k(a))) - \log (e^{\mu} \mu (1-f^k(a)) ) \big| \\
&=& \lambda - \mu.
\end{eqnarray*}\end{small}Putting everything together, we obtain the announced equality (\ref{fund-eq}).
\end{proof}

\medskip

The last special feature of $\mathrm{PL}_+^{\Delta}([0,1])$ is the well-known fact that the Mather invariant 
can never be (the class of) a rotation (this is closely related to the fact that the centralizer of an element 
$f$ in $\mathrm{PL}_+^{\Delta}([0,1])$ is a finite extension of the group generated by $f$, and it is infinite cyclic). 
Indeed, this follows from the fundamental equality (\ref{fund-eq}), which in its turn implies that 
$$\mathrm{var} (\log DM_f) \geq 2 \, \big[ | \log Df (0) | + | \log Df (1) | \big];$$
see \cite{matucci} for an alternative (less quantitative) argument.

\begin{ex} Let $f$ be the piecewise-affine homeomorphisms considered in Example 4.3 from \cite{EN}. 
It can be readily checked that $\, \dist (f) = \lambda - \mu \,$, while $\, \var(\log DM_f ) = 2 \, (\lambda - \mu). \,$ 
\end{ex}


\subsection{A remark concerning the Mather homomorphism} 

Another remarkable object introduced by Mather is a group homomorphism from 
the group of $C^{1+\mathrm{bv}}$ diffeomorphisms of a 1-manifold into $\mathbb{R}$. (The motivation was to prove the non-simplicity of such 
a group; see \cite{mather}). Although the original construction concerns diffeomorphisms of the real line with compact support, it 
also applies to the circle and the interval. To be more concrete, given a diffeomorphism $f \in \mathrm{Diff}_+^{1+\mathrm{bv}} ([0,1])$, 
we denote $\mu_f$ the (finite) signed measure induced by the Riemann-Stieltjes integration with respect to $\log (Df)$. (The fact that 
$\log (Df)$ has bounded variation implies that this integration is well defined.) The measure $\mu_f$ has a unique decomposition 
$$\mu_f = \mu_f^{ac} + \mu_f^{s},$$
where $\mu_f^{ac}$ (resp. $\mu_f^s$) is absolutely continuous (resp. totally singular) with respect to the Lebesgue measure. 
We then let
$$\phi_M (f) := \int_0^1 d \mu_f^{ac}.$$
It is straighhforward to check that $\phi_M$ defines a continuous group homomorphism from $\mathrm{Diff}^{1+\mathrm{bv}}_+([0,1])$ onto 
$\mathbb{R}$ (see \cite{mather} for further details).

\medskip

\begin{prop} If a diffeomorphism $f \in \mathrm{Diff}_+^{1+\mathrm{bv},\Delta} ([0,1])$ has parabolic fixed points and its image under 
the Mather homomorphism $\phi_M$ is nonzero, then its Mather invariant $M_f$ is nontrivial.
\end{prop}

\noindent{\bf Proof.} If $M_f$ is trivial for $f \in \mathrm{Diff}_+^{1+\mathrm{bv},\Delta} ([0,1])$ with parabolic fixed points then, 
according to the fundamental inequality (\ref{fund-eq}), one has $\dist (f) = 0$. 
By \cite{EN}, there exists a sequence of $C^{1+\mathrm{bv}}$ diffeomorphisms $h_n$ of $[0,1]$ such that 
$h_n f h_n^{-1}$ converges to the identity in the $C^{1+\mathrm{bv}}$ topology. By the continuity of $\phi_M$, this implies that 
$\phi_M (h_n f h_n^{-1})$ converges to zero. However, since $\phi_M$ is a group homomorphism, for each $n$ we 
have $\phi_M (h_n f h_n^{-1}) = \phi_M (f)$. Therefore, if $M_f$ is trivial, then $\phi_M (f) = 0$.
$\hfill\square$

\medskip

\begin{qs}
{\em Is it possible to have $\phi_M(f) = 0 \neq \phi_M(g)$ for two $C^2$ diffeomorphisms $f$ and $g$ that are $C^1$-conjugate~? 
Compare \cite[Theorem D]{EN}, which establishes the invariance of the asymptotic variation under $C^1$ conjugacy.}
\end{qs}


\section{Appendix II: A remark concerning $C^1$ vector fields}
\label{section-sternberg}

The goal here is to give an example of a $C^{1+\mathrm{ac}}$ diffeomorphism of the interval with an hyperbolic 
fixed point which is not $C^1$ conjugate to its linear part close to that point.

Let us recall that Sternberg gave in \cite{sternberg} an example of a hyperbolic germ of $C^1$ 
diffeomorphism that is not $C^1$ (even bi-Lipschitz) linearizable, namely,
$$x \mapsto e^{\lambda} \, x \left( 1 - \frac{1}{\log (x)} \right),$$
where $\lambda < 0$. Inspired on this, and following \cite[Exercise 4.1.12]{libro}, 
let us consider a $C^1$ vector field on $[0,1)$ that vanishes only at $0$ and 
satisfies on a neighborhood of $0$ the equality
$$\mathcal{X}(x) := \lambda x \left( 1 - \frac{1}{\log(x)} \right) \esp 
\frac{\partial}{\partial x}.$$ 
It is easy to see that $\mathcal{X}$ is hyperbolic at the origin, with linear part $\lambda x \frac{\partial}{\partial x}$. 
We claim, however, that $\mathcal{X}$ is not $C^1$ linearizable. 

To show this, we first claim that if $f$ denotes the time-1 
map of the flow of  $\mathcal{X}$, then for every $x > 0$ close enough to the origin, one has
\begin{equation}
f(x) = e^{\lambda} \esp x \esp \left( \frac{1 - \log(x)}{1 - \log(f(x))} \right).
\label{relacion}
\end{equation}
Indeed, let us fix such an $x > 0$, and let us denote by $h(t)$ the solution of 
$$\frac{dh}{dt}(t) = \lambda \, h(t) \left( 1 - \frac{1}{\log(h(t))} \right), 
\quad \mbox{ with} \quad h(0) = x.$$ 
If we put $\varphi(t) := \log(h(t))$, then we have
\, $\varphi' = \lambda \left( 1 - \frac{1}{\varphi} \right),$\,
and so
$$(\varphi^{-1})'(t) = \frac{1}{\lambda(1-1/t)} = \frac{t}{\lambda(t-1)}.$$
Since $h(0) = x$, we have $\varphi (0) = \log(x)$, hence
$\varphi^{-1}(\log(x)) = 0$. Therefore,
$$\varphi^{-1}(t) = - \int_{\log(x)}^{t} \frac{s}{\lambda(1-s)} 
\esp ds = \frac{t}{\lambda} - \frac{\log(x)}{\lambda} + 
\frac{1}{\lambda} \log \left( \frac{1-t}{1-\log(x)} \right).$$
Since $\varphi^{-1}(\log(h(t))) = t$, this gives
$$t = \frac{\log(h(t))}{\lambda} - \frac{\log(x)}{\lambda} + 
\frac{1}{\lambda}\log \left( \frac{1 - \log(h(t)) }{1 - \log(x)} \right),$$
and so
$$\log(h(t)) = \lambda t + \log(x) -  
\log \left( \frac{1 - \log(h(t))}{1 - \log(x)} \right).$$
Since $h(1) = f(x)$, we have
$$\log(f(x)) = \lambda + \log(x) - 
\log \left( \frac{1 - \log(f(x))}{1 - \log(x)} \right),$$
which proves (\ref{relacion}).

Let us now suppose by contradiction that $\mathcal{X}$ is $C^1$ conjugate to its linear part. If it was, then $f$ would also 
be $C^1$ conjugate to its linear part. However, as we next show, this is not the case. 
Indeed, from (\ref{relacion}), one easily concludes that
\begin{eqnarray*}
\frac{f^k (x)}{e^{\lambda k} x}  
&=& \frac{f(x)}{e^{\lambda} x} \cdot \frac{f^2(x)}{e^{\lambda} f(x)} \ldots 
\frac{f^k(x)}{e^{\lambda} f^{k-1}(x)}\\
&=& \frac{1 - \log(x)}{1 - \log(f(x))} \cdot \frac{1 - \log(f(x))}{1 - \log(f^2(x))} 
\cdots \frac{1 - \log(f^{k-1}(x))}{1 - \log(f^{k}(x))} 
\,\,\,=\,\,\, \frac{1 - \log(x)}{1 - \log(f^k(x))}.
\end{eqnarray*}
The right-hand-side expression converges to zero as $k$ goes to infinity (since $f^k(x)$ converges to the origin). However, if $f$ 
was conjugated to $x \mapsto e^{\lambda} x$ by some bi-Lipschitz homeomorphism $\phi$ with bi-Lipschitz constant $M$, then from 
$f^k(x) = \phi (e^{\lambda k} \phi^{-1}(x))$ one would obtain a.e. close to the origin:
\, $Df^k \geq e^{\lambda k} / M^2.$ \,
By integration, this would yield 
$$f^k (x) \geq \frac{e^{\lambda k}x}{M^2},$$
which contradicts the convergence of \, $f^k (x) / e^{\lambda k} x$ \, to zero.

To close this discussion, we claim that $f$ can be explicitly conjugated to its linear part by a non bi-Lipschitz map. Indeed, 
letting $\Phi$ be so that $\Phi (x) := x \, (1-\log(x) )$ close to the origin (and extending it in an equivariant way), equality 
(\ref{relacion}) may be read as \, $\Phi (f(x)) = e^{\lambda} \, \Phi(x),$ \, 
which is the announced conjugacy relation. Notice that \, $D \Phi (x) = -\log (x).$ \,
Using this relation (or by a direct analysis), a straightforward computation shows that 
$f$ is of class $C^{1+\mathrm{ac}}$. It is 
worth to stress that $f$ is not of class $C^{1+\alpha}$ for any $\alpha > 0$, because the Sternberg-Yoccoz linearization 
theorem still holds in this setting \cite{MW}.
\vspace{0.5cm}


\noindent{\bf Acknowledgments.}  Both authors were funded by the ANR Project GROMEOV, and 
strongly thank the organizers of the Workshop ``Ordered groups and rigidity in dynamics and topology'' 
at BIRS /  Casa Matem\'atica de Oaxaca (June 2019), where the main result of this article was obtained. 
A. Navas was also funded by the FONDECYT project 1200114; he would also like to acknowledge 
support from UNAM via the project FORDECYT 265667 and the PREI of the DGAPA.

Conversations with Christian Bonatti, Bassam Fayad, \'Etienne Ghys, Alejandro Kocsard  
and Raphael Krikorian were very useful in the preparation of this paper.

\vspace{0.5cm}


\begin{footnotesize}

\vspace{0.42cm}

\noindent {\bf H\'el\`ene Eynard-Bontemps} \hfill{\bf Andr\'es Navas}

\noindent Institut Fourier \hfill{ Mathematics and Computer Science Dept.}

\noindent  Universit\'e Grenoble Alpes \hfill{ University of Santiago of  Chile}

\noindent 100 rue des Math\'ematiques \hfill{Alameda Bernardo O'Higgins 3363}

\noindent 38610 Gi\`eres, France \hfill{Estaci\'on Central, Santiago, Chile} 

\noindent helene.eynard-bontemps@univ-grenoble-alpes.fr \hfill{andres.navas@usach.cl}

\end{footnotesize}

\end{document}